\documentclass[amsfonts,12pt,reqno]{amsart}%
\usepackage{amsmath}
\usepackage{amsfonts}
\usepackage{amssymb}
\usepackage{graphicx}
\usepackage{setspace}
\allowdisplaybreaks
\usepackage[top=3cm, bottom=3cm, left=3cm, right=3cm]{geometry}%
\providecommand{\U}[1]{\protect\rule{.1in}{.1in}}

\newtheorem{theorem}{Theorem}[section]

\newtheorem{corollary}[theorem]{Corollary}

\newtheorem{definition}[theorem]{Definition}
\newtheorem{example}[theorem]{Example}

\newtheorem{lemma}[theorem]{Lemma}

\newtheorem{proposition}[theorem]{Proposition}
\newtheorem{remark}[theorem]{Remark}

\numberwithin{equation}{section}

\begin{document}
\title[Conformal anti-Invariant submersions]{Conformal anti-invariant submersions from almost Hermitian manifolds}
\author[Akyol \lowercase{and} \c{S}ahin]{M\lowercase{ehmet} A\lowercase{kif} AKYOL, B\lowercase{ayram} \c{S}AH\.{I}N}
\address{Department of Mathematics, Faculty of Science and Arts,B\.{I}ng\"{o}l University, B\.{I}ng\"{o}l, Turkey}
\address{Department of Mathematics, Faculty of Science and Arts,\.{I}n\"{o}n\"{u} University,
44280, Malatya, Turkey}

\email{makyol@bingol.edu.tr}
\email{bayram.sahin@inonu.edu.tr}
\maketitle

\doublespace

\section*{\textbf{Abstract}}

We introduce conformal anti-invariant submersions from almost
Hermitian manifolds onto Riemannian manifolds. We give  examples, investigate
the geometry of foliations which are arisen from the definition of a conformal
submersion and find necessary and sufficient conditions for a conformal
anti-invariant submersion to be totally geodesic. We also check the harmonicity of such submersions and show that the total space has certain product structures.  Moreover, we obtain
curvature relations between the base space and the total space, and find geometric implications of these relations.
\newline

{\small \noindent {{\bf Keywords:} Riemannian submersion, Anti-invariant submersion, Conformal
submersion, conformal anti-invariant submersion.}\newline
}

\makeatother{\section{\textbf{Introduction}}}
One of the main method to compare two manifolds and transfer certain structures from a manifold to another manifold is to define appropriate smooth maps between them. Given two manifolds, if the rank of a differential map is equal to the dimension of the source manifold, then such maps are called immersions and if the rank of  a differential map is equal to the target manifold, then such maps are called submersions. Moreover, if these maps are isometry between manifolds, then the immersion is called isometric immersion (Riemannian submanifold) and the submersion is called Riemannian submersion. Riemannian submersions  between Riemannian manifolds were studied by O'Neill \cite{O} and Gray \cite{Gray}, for recent developments on the geometry of Riemannian submanifolds and Riemannian submersions, see:\cite{Chen} and \cite{FIP}, respectively.

On the other hand, as a generalization of Riemannian submersions,
horizontally conformal submersions are defined as
follows \cite{BW}: Suppose that $(M, g_{_M})$ and $(B,
g_{_B})$ are Riemannian manifolds and $F:M\longrightarrow B$ is a
smooth submersion, then $F$ is called a horizontally conformal
submersion, if there is a positive function $\lambda$ such that
$$\lambda^2\, g_{_M}(X, Y)=g_{_B}(F_*X, F_*Y)$$
for every $X, Y \in  \Gamma((ker F_*)^\perp)$. It is obvious that
every Riemannian submersion is a particular horizontally conformal
submersion with $\lambda=1$. We note that horizontally conformal
submersions are special horizontally conformal maps which were
introduced independently by Fuglede \cite{Fuglede} and Ishihara
\cite{Ishihara}. We also note that  a
 horizontally conformal submersion $F: M \longrightarrow B$
 is said to be horizontally homothetic if the gradient of its
 dilation $\lambda$ is vertical, i.e.,
 \begin{equation}
 \mathcal{H}(grad \lambda)=0 \label{eq:1.2}
 \end{equation}
 at $p \in M,$ where $\mathcal{H}$ is the projection on the
 horizontal space $(
    ker {F_*}_p)^{\perp}$. For conformal submersions, see: \cite{BW}, \cite{Chinea2}, \cite{Chinea}, \cite{Chinea3}, \cite{FIP} and \cite{G}.

 One can see that  Riemannian submersions are very special maps comparing with conformal submersions. Although conformal maps does not preserve distance between points contrary to isometries, they preserve angles between vector fields. This property enables one to transfer certain properties of a manifold  to another manifold by deforming such properties.

A submanifold of a complex manifold is a complex (invariant) submanifold if the tangent space of the submanifold at each point is invariant with respect to the almost complex structure of the K\"{a}hler manifold. Besides complex submanifolds of a complex manifold, there is another important class of submanifolds called totally real submanifolds. A totally real submanifold of a complex manifold is a submanifold of such that the almost complex structure of ambient manifold carries the tangent space of the submanifold at each point into its normal space. Many authors have studied totally real submanifolds in various ambient manifolds and many interesting results were obtained, see (\cite{Chen}, page:322) for a survey on all these results..

 As analogue of holomorphic submanifolds, holomorphic submersions were introduced by Watson \cite{Watson} in seventies by using the notion of almost complex map. This notion has been extended to other manifolds, see\cite{FIP} for holomorphic submersions and their extensions to other manifolds. The main property of such maps is that the vertical distributions and the horizontal distributions of such maps are invariant with respect to almost complex map.  Therefore, the second author of the present paper considered a new submersion defined on an almost Hermitian manifold such that the vertical distribution is anti-invariant with respect to almost complex structure \cite{Þ1}. He showed that such submersions have rich geometric properties and they are useful for investigating the geometry of the total space. This new class of submersions which is called anti-invariant submersions can be seen as an analogue of totally real submanifolds in the submersion theory. Anti-invariant submersions have been also studied for different total manifolds, see: \cite{AF}, \cite{Lee} and \cite{ME}.

 As a generalization of holomorphic submersions, conformal holomorphic submersions were studied by Gudmundsson and Wood \cite{GW}. They obtained necessary and sufficient conditions for conformal holomorphic submersions to be a harmonic morphism, see also \cite{Chinea2}, \cite{Chinea} and \cite{Chinea3} for the harmonicity of conformal holomorphic submersions.

In this paper, we study conformal anti-invariant submersions as a generalization of anti-invariant Riemannian submersions and investigate the geometry of the total space and the base space for the existence of such submersions. The paper is organized as follows. In the second section, we gather main notions and formulas for other sections. In section 3, we introduce conformal anti-invariant submersions from almost Hermitian manifolds onto Riemannian manifolds, give examples and investigates the geometry of leaves of the horizontal distribution and the vertical distribution. In section 4, we find necessary and sufficient conditions for a conformal anti-invariant submersion to be harmonic and totally geodesic, respectively. In section 5, we show that there are certain product structures on the total space of a conformal anti-invariant submersion. In section 6, we study curvature relations between the total space and the base space, find several inequalities and obtain new results when the inequality becomes the equality.
\section{\textbf{Preliminaries}}

In this section, we define almost Hermitian manifolds, recall the notion of
({\it horizontally}) conformal submersions between Riemannian manifolds and
give a brief review of basic facts of ({\it horizontally}) conformal submersions.

Let $(M,g)$ be an almost Hermitian manifold. This means \cite{YK} that $M$
admits a tensor field $J$ of type $(1,1)$ on $M$ such that, $\forall
X,Y\in\Gamma(TM),$ we have%
\begin{equation}
J^{2}=-I,\mbox{ }g(X,Y)=g(JX,JY). \label{e.q:2.1}%
\end{equation}

An almost Hermitian manifold $M$ is called K\"{a}hler manifold if%
\begin{equation}
(\nabla_{X}J)Y=0,\mbox{ }\forall X,Y\in\Gamma(TM), \label{e.q:2.2}%
\end{equation}
where $\nabla$ is the Levi-Civita connection on $M$.

Conformal submersions belong to a wide class of conformal maps that we are going to recall their definition, but we will not study such maps in this paper.

\begin{definition}
\textrm{(\cite{BW})}Let $\varphi:(M^{m},g)\rightarrow(N^{n},h)$ be a smooth
map between Riemannian manifolds, and let $x\in M$. Then $\varphi$ is called
horizontally weakly conformal or semiconformal at $x$ if either

(i) $d\varphi_{x}=0,$ or

(ii) $d\varphi_{x}$ maps the horizontal space $\mathcal{H}_{x}=\{\ker
(d\varphi_{x})\}^{\perp}$ conformally onto $T_{\varphi(x)}N,$ i.e.,
$d\varphi_{x}$ is surjective and there exists a number $\Lambda(x)\neq0$ such
that%
\begin{equation}
h(d\varphi_{x}(X),d\varphi_{x}(Y))=\Lambda(x)g(X,Y)\mbox{ }(X,Y\in
\mathcal{H}_{x}). \label{e.q:2.3}%
\end{equation}

\end{definition}

Note that we can write the last equation more succinctly as%
\[
(\varphi^{\ast}h)_{x}\mid_{\mathcal{H}_{x}\times\mathcal{H}_{x}}%
=\Lambda(x)g_{x}\mid_{\mathcal{H}_{x}\times\mathcal{H}_{x}}.
\]
With the above definition of critical point, a point $x$ is of type (i) in
Definition 2.1 if and only if it is a critical point of $\varphi;$ we shall call
a point of type (ii) a \textit{regular point. }At a critical point,
$d\varphi_{x}$ has rank $0;$ at a regular point, $d\varphi_{x}$ has rank $n$
and $\varphi$ is a submersion. The number $\Lambda(x)$ is called the
\textit{square dilation }(of $\varphi$ at $x$); it is necessarily
non-negative; its square root $\lambda(x)=\sqrt{\Lambda(x)}$ is called the
dilation (of $\varphi$ at $x$). The map $\varphi$ is called\textit{
horizontally weakly conformal} or \textit{semiconformal} (on $M$) if it is
horizontally weakly conformal at every point of $M$. It is clear that  if $\varphi$
has no critical points, then we call it a (\textit{horizontally}) conformal submersion.

Next, we recall the following definition from \cite{G}.
Let $\pi:M\rightarrow N$ be a submersion. A vector field $E$ on $M$ is said to
be projectable if there exists a vector field $\breve{E}$ on $N,$ such that
$d\pi(E_{x})=\breve{E}_{\pi(x)}$ for all $x\in M$. In this case $E$ and
$\breve{E}$ are called $\pi-$related. A horizontal vector field $Y$ on $(M,g)$
is called basic, if it is projectable. It is a well known fact that if $\breve{Z}$ is a vector field on $N,$ then there exists a unique basic vector
field $Z$ on $M$, such that $Z$ and $\breve{Z}$ are $\pi-$related. The vector
field $Z$ is called the horizontal lift of $\breve{Z}$.

The fundamental tensors of a submersion were introduced in \cite{O}. They play
a similar role to that of the second fundamental form of an immersion. More precisely,
O'Neill's tensors $T$ and $A$ defined for vector fields $E,F$ on $M$ by%
\begin{equation}
A_{E}F=\mathcal{V}\nabla_{\mathcal{H}E}\mathcal{H}F+\mathcal{H}\nabla
_{\mathcal{H}E}\mathcal{V}F \label{e.q:2.4}%
\end{equation}%
\begin{equation}
T_{E}F=\mathcal{H}\nabla_{\mathcal{V}E}\mathcal{V}F+\mathcal{V}\nabla
_{\mathcal{V}E}\mathcal{H}F \label{e.q:2.5}%
\end{equation}
where $\mathcal{V}$ and $\mathcal{H}$ are the vertical and horizontal
projections (see \cite{FIP}). On the other hand, from (\ref{e.q:2.4}) and
(\ref{e.q:2.5}), we have%
\begin{equation}
\nabla_{V}W=T_{V}W+\hat{\nabla}_{V}W \label{e.q:2.6}%
\end{equation}%
\begin{equation}
\nabla_{V}X=\mathcal{H}\nabla_{V}X+T_{V}X \label{e.q:2.7}%
\end{equation}%
\begin{equation}
\nabla_{X}V=A_{X}V+\mathcal{V}\nabla_{X}V \label{e.q:2.8}%
\end{equation}%
\begin{equation}
\nabla_{X}Y=\mathcal{H}\nabla_{X}Y+A_{X}Y \label{e.q:2.9}%
\end{equation}
for $X,Y\in\Gamma((\ker\pi_{\ast})^{\perp})$ and $V,W\in\Gamma(\ker\pi_{\ast
}),$ where $\hat{\nabla}_{V}W=\mathcal{V}\nabla_{V}W$. If $X$ is basic, then
$\mathcal{H}\nabla_{V}X=A_{X}V$.

It is easily seen that for $x\in M,$ $X\in\mathcal{H}_{x}$ and $V \in \mathcal{V}%
_{x}$ the linear operators $T_{V},A_{X}:T_{x}M\rightarrow T_{x}M$ are
skew-symmetric, that is%
\[
-g(T_{V}E,F)=g(E,T_{V}F)\text{ and }-g(A_{X}E,F)=g(E,A_{X}F)
\]
for all $E,F\in T_{x}M$. We also see that the restriction of $T$ to the
vertical distribution $T\mid_{\mathcal{V\times V}}$ is exactly the second
fundamental form of the fibres of $\pi$. Since $T_{V}$ is skew-symmetric we get:
$\pi$ has totally geodesic fibres if and only if $T\equiv0$. For the special
case when $\pi$ is horizontally conformal we have the following:

\begin{proposition}
\textrm{(\cite{G})} Let $\pi:(M^{m},g)\rightarrow(N^{n},h)$ be a horizontally
conformal submersion with dilation $\lambda$ and $X,Y$ be horizontal vectors,
then%
\begin{equation}
A_{X}Y=\frac{1}{2}\{\mathcal{V}\left[  X,Y\right]  -\lambda^{2}%
g(X,Y)\operatorname{grad}_{\mathcal{V}}(\frac{1}{\lambda^{2}})\}.
\label{e.q:2.10}%
\end{equation}

\end{proposition}

We see that the skew-symmetric part of $A\mid_{\mathcal{H\times H}}$ measures
the obstruction integrability of the horizontal distribution $\mathcal{H}$.

We now recall the following curvature relations for a conformal submersion from \cite{Gromoll-Klingenberg-Meyer} and \cite{G}.

\begin{theorem}
Let $m>n\geq2$ and $(M^{m},g,\nabla,R)$, $(N^{n},h,\nabla^{N},R^{N})$ be two
Riemannian manifolds with their Levi-Civita connections and the corresponding
curvature tensors. Let $\pi:(M,g)\rightarrow(N,h)$ be a horizontally conformal
submersion, with dilation $\lambda:M\rightarrow%
\mathbb{R}
^{+}$ and let $R^{\mathcal{V}}$ be the curvature tensor of the fibres of
$\pi$. If $X,Y,Z,H$ are horizontal and $U,V,W,F$ vertical vectors, then%
\begin{equation}
g(R(U,V)W,F)=g(R^{\mathcal{V}}(U,V)W,F)+g(T_{U}W,T_{V}F)-g(T_{V}W,T_{U}F),
\label{e.q:2.11}%
\end{equation}%
\begin{equation}
g(R(U,V)W,X)=g((\nabla_{U}T)_{V}W,X)-g((\nabla_{V}T)_{U}W,X), \label{e.q:2.12}%
\end{equation}%
\begin{align}
g(R(U,X)Y,V)  &  =g((\nabla_{U}A)_{X}Y,V)+g(A_{X}U,A_{Y}V)\label{e.q:2.13}\\
&  -g((\nabla_{X}T)_{U}Y,V)-g(T_{V}Y,T_{U}X)\nonumber\\
&  +\lambda^{2}g(A_{X}Y,U)g(V,\operatorname{grad}_{\mathcal{V}}(\frac
{1}{\lambda^{2}})),\nonumber
\end{align}%
\begin{align}
g(R(X,Y)Z,H)  &  =\frac{1}{\lambda^{2}}h(R^{N}(\breve{X},\breve{Y})\breve
{Z},\breve{H})+\frac{1}{4}[g(\mathcal{V}[X,Z],\mathcal{V}%
[Y,H])\label{e.q:2.14}\\
&  -g(\mathcal{V}[Y,Z],\mathcal{V}[X,H])+2g(\mathcal{V}[X,Y],\mathcal{V}%
[Z,H])]\nonumber\\
&  +\frac{\lambda^{2}}{2}[g(X,Z)g(\nabla_{Y}\operatorname{grad}(\frac
{1}{\lambda^{2}}),H)-g(Y,Z)g(\nabla_{X}\operatorname{grad}(\frac{1}%
{\lambda^{2}}),H)\nonumber\\
&  +g(Y,H)g(\nabla_{X}\operatorname{grad}(\frac{1}{\lambda^{2}}%
),Z)-g(X,H)g(\nabla_{Y}\operatorname{grad}(\frac{1}{\lambda^{2}}%
),Z)]\nonumber\\
&  +\frac{\lambda^{4}}{4}[(g(X,H)g(Y,Z)-g(Y,H)g(X,Z))\parallel\operatorname{grad}%
(\frac{1}{\lambda^{2}})\parallel^{2}\nonumber\\
&  +g(X(\frac{1}{\lambda^{2}})Y-Y(\frac{1}{\lambda^{2}})X,H(\frac{1}%
{\lambda^{2}})Z-Z(\frac{1}{\lambda^{2}})H)].\nonumber
\end{align}

\end{theorem}

We also recall the notion of harmonic maps between Riemannian manifolds. Let
$(M,g_{M})$ and $(N,g_{N})$ be Riemannian manifolds and suppose that
$\varphi:M\rightarrow N$ is a smooth map between them. Then the differential
of $\varphi_{\ast}$ of $\varphi$ can be viewed a section of the bundle
$Hom(TM,\varphi^{-1}TN)\rightarrow M,$ where $\varphi^{-1}TN$ is the pullback
bundle which has fibres $(\varphi^{-1}TN)_{p}=T_{\varphi(p)}N,$ $p\in M$.
$Hom(TM,\varphi^{-1}TN)$ has a connection $\nabla$ induced from the
Levi-Civita connection $\nabla^{M}$ and the pullback connection. Then the
second fundamental form of $\varphi$ is given by%
\begin{equation}
(\nabla\varphi_{\ast})(X,Y)=\nabla_{X}^{\varphi}\varphi_{\ast}(Y)-\varphi
_{\ast}(\nabla_{X}^{M}Y) \label{e.q:2.15}%
\end{equation}
for $X,Y\in\Gamma(TM),$ where $\nabla^{\varphi}$ is the pullback connection.
It is known that the second fundamental form is symmetric. A smooth map
$\varphi:(M,g_{M})\rightarrow(N,g_{N})$ is said to be harmonic if
$trace(\nabla\varphi_{\ast})=0$. On the other hand, the tension field of
$\varphi$ is the section $\tau(\varphi)$ of $\Gamma(\varphi^{-1}TN)$ defined
by%
\begin{equation}
\tau(\varphi)=\operatorname{div}\varphi_{\ast}=\sum_{i=1}^{m}(\nabla
\varphi_{\ast})(e_{i},e_{i}), \label{e.q:2.16}%
\end{equation}
where $\{e_{1},...,e_{m}\}$ is the orthonormal frame on $M$. Then it follows
that $\varphi$ is harmonic if and only if $\tau(\varphi)=0,$ for details, see
\cite{BW}.

Finally, we recall the following lemma from \cite{BW}.
\begin{lemma}
\label{lem1}(Second fundamental form of an HC submersion) Suppose that
$\varphi:M\rightarrow N$ is a horizontally conformal submersion. Then, for any
horizontal vector fields $X,Y$ and vertical vector fields $V,W,$ we have

$
\begin{array}
[c]{lll}%
(i) & \nabla d\varphi(X,Y)= & X(\ln\lambda)d\varphi(Y)+Y(\ln\lambda
)d\varphi(X)-g(X,Y)d\varphi(\operatorname{grad}\ln\lambda);\\
(ii) & \nabla d\varphi(V,W)= & -d\varphi(A_{V}^{\mathcal{V}}W);\\
(iii) & \nabla d\varphi(X,V)= & -d\varphi(\nabla_{X}^{M}V)=d\varphi
((A^{\mathcal{H}})_{X}^{\ast}V).
\end{array}
$

Here $(A^{\mathcal{H}})_{X}^{\ast}$ is the adjoint of $A_{X}^{\mathcal{H}}$
characterized by%
\[
\langle(A^{\mathcal{H}})_{X}^{\ast}E,F\rangle=\langle E,A_{X}^{\mathcal{H}%
}F\rangle\mbox{ }(E,F\in\Gamma(TM)).
\]

\end{lemma}

\section{\textbf{Conformal Anti-invariant Submersions}\label{sect-prel}}

In this section, we define conformal anti-invariant submersions from an almost
Hermitian manifold onto a Riemannian manifold and investigate the effect of the existence of conformal anti-invariant submersions on the source manifold and the target manifold. But we first present the following notion.

\begin{definition}
\label{def3.1}Let $M$ be a complex $m$-dimensional almost Hermitian manifold
with Hermitian metric $g$ and almost complex structure $J$ and $N$ be a
Riemannian manifold with Riemannian metric $g^{\prime}$. A  horizontally
conformal submersion  $F:(M^{m},g)\rightarrow(N^{n},g^{\prime})$ with dilation $\lambda$ is a called conformal anti-invariant submersion if the distribution $\ker F_{\ast}$ is anti-invariant with respect to $J,$ i.e., $J(\ker
F_{\ast})\subseteq(\ker F_{\ast})^{\perp}$.
\end{definition}

Let $F:(M,g,J)\rightarrow(N,g^{\prime})$ be a conformal anti-invariant
submersion from an almost Hermitian manifold $(M,g,J)$ to a Riemannian
manifold $(N,g^{\prime})$. First of all, from Definition 3.1, we have $J(\ker
F_{\ast})^{\perp}\cap\ker F_{\ast}\neq\{0\}$. We denote the complementary
orthogonal distribution to $J(\ker F_{\ast})$ in $(\ker F_{\ast})^{\perp}$ by
$\mu$. Then we have%
\begin{equation}
(\ker F_{\ast})^{\perp}=J(\ker F_{\ast})\oplus\mu. \label{e.q:3.1}%
\end{equation}
It is easy to see that $\mu$ is an invariant distribution of $(\ker F_{\ast
})^{\perp},$ under the endomorphism $J$. Thus, for $X\in\Gamma((\ker F_{\ast
})^{\perp}),$ we have%
\begin{equation}
JX=BX+CX, \label{e.q:3.2}%
\end{equation}
where $BX\in\Gamma(\ker F_{\ast})$ and $CX\in\Gamma(\mu)$. On the other hand,
since $F_{\ast}((\ker F_{\ast})^{\perp})=TN$ and $F$ is a conformal submersion,
using (\ref{e.q:3.2}) we derive $\frac{1}{\lambda^{2}}g^{\prime}(F_{\ast
}JV,F_{\ast}CX)=0,$ for every $X\in\Gamma((\ker F_{\ast})^{\perp})$ and
$V\in\Gamma(\ker F_{\ast}),$ which implies that%
\begin{equation}
TN=F_{\ast}(J(\ker F_{\ast}))\oplus F_{\ast}(\mu). \label{e.q:3.3}%
\end{equation}
\begin{example} Every anti-invariant Riemannian submersion is a conformal anti-invariant submersion with $\lambda=I$, where $I$ is the identity function.
\end{example}
We say that a conformal anti-invariant submersion is proper if $\lambda\neq I$.
We now present an example of a proper conformal anti-invariant submersion. In the following $R^{2m}$ denotes the Euclidean $2m$-space with the standard
metric. An almost complex structure $J$ on $R^{2m}$ is said to be compatible
if $(R^{2m},J)$ is complex analytically isometric to the complex number space
$C^{m}$ with the standard flat K\"{a}hlerian metric. We denote by $J$ the
compatible almost complex structure on $R^{2m}$ defined by%
\[
J(a^{1},...,a^{2m})=(-a^{2},a^{1},...,-a^{2m},a^{2m-1}).
\]

\begin{example}
\label{exm1}Let $F$ be a map defined by
$$
\begin{array}{cccc}
  F: & R^4             & \longrightarrow & R^2\\
     & (x_1,x_2,x_3,x_4) &             & (e^{x_{3}}\sin x_{4},e^{x_{3}}\cos x_{4}).
\end{array}
$$
 Then
 $F$ is a conformal anti-invariant
submersion with $\lambda=e^{x_{3}}$.
\end{example}

\begin{lemma}
Let $F$ be a conformal anti-invariant submersion from a K\"{a}hler manifold
$(M,g,J)$ to a Riemannian manifold $(N,g^{\prime})$. Then we have%
\begin{equation}
g(CY,JV)=0 \label{e.q:3.4}%
\end{equation}
and%
\begin{equation}
g(\nabla_{X}CY,JV)=-g(CY,JA_{X}V) \label{e.q:3.5}%
\end{equation}
for $X,Y\in\Gamma((\ker F_{\ast})^{\perp})$ and $V\in\Gamma(\ker F_{\ast})$.
\end{lemma}

\begin{proof}
For $Y\in\Gamma((\ker F_{\ast})^{\perp})$ and $V\in\Gamma(\ker F_{\ast})$, since $BY\in\Gamma(\ker F_{\ast})$ and $JV\in\Gamma((\ker F_{\ast})^{\perp
})$, using (\ref{e.q:2.1}), we get  (\ref{e.q:3.4}). Now, using
(\ref{e.q:3.4}), (\ref{e.q:2.2}) and (\ref{e.q:2.8}) we obtain
$$
g(\nabla_{X}CY,JV)=-g(CY,JA_{X}V)-g(CY,J\mathcal{V}\nabla_{X}V).
$$
Since $J\mathcal{V}\nabla_{X}V\in\Gamma(J\ker F_{\ast}),$ we obtain
(\ref{e.q:3.5}).
\end{proof}

We now study the integrability of the distribution $(\ker F_{\ast})^{\perp}$
and then we investigate the geometry of leaves of $\ker F_{\ast}$ and ($\ker
F_{\ast})^{\perp}$. We note that it is known that the distribution $\ker
F_{\ast}$ is integrable.

\begin{theorem}
\label{teo1}Let $F$ be a conformal anti-invariant submersion from a K\"{a}hler
manifold $(M,g,J)$ to a Riemannian manifold $(N,g^{\prime})$. Then the
following assertions are equivalent to each other;

$%
\begin{array}
[c]{ll}%
a) & (\ker F_{\ast})^{\perp}\text{ is integrable,}%
\end{array}
$

$%
\begin{array}
[c]{lll}%
b) & \frac{1}{\lambda^{2}}g^{\prime}(\nabla_{Y}^{F}F_{\ast}CX-\nabla_{X}%
^{F}F_{\ast}CY,F_{\ast}JV)= & g(A_{X}BY-A_{Y}BX,JV)\\
& &-g(\mathcal{H}%
\operatorname{grad}\ln\lambda,CY)g(X,JV)\\
&  & +g(\mathcal{H}\operatorname{grad}\ln\lambda
,CX)g(Y,JV)\\
& &-2g(CX,Y)g(\mathcal{H}\operatorname{grad}\ln\lambda,JV)
\end{array}
$

for $X,Y\in\Gamma((\ker F_{\ast})^{\perp})$ and $V\in\Gamma(\ker F_{\ast})$.
\end{theorem}

\begin{proof}
For $Y\in\Gamma((\ker F_{\ast})^{\perp})$ and $V\in\Gamma(\ker F_{\ast}),$ we
see from Definition \ref{def3.1}, $JV\in\Gamma((\ker F_{\ast})^{\perp})$ and
$JY\in\Gamma(\ker F_{\ast}\oplus\mu)$. Thus using (\ref{e.q:2.1}) and
(\ref{e.q:2.2}), for $X\in\Gamma((\ker F_{\ast})^{\perp})$ we get%
$$
g(\left[  X,Y\right]  ,V)=g(\nabla_{X}JY,JV)-g(\nabla_{Y}JX,JV).
$$
Then from (\ref{e.q:3.2}) we have%
\begin{align*}
g(\left[  X,Y\right]  ,V)  &  =g(\nabla_{X}BY,JV)+g(\nabla_{X}CY,JV)\\
&  -g(\nabla_{Y}BX,JV)-g(\nabla_{Y}CX,JV).
\end{align*}
Since $F$ is a conformal submersion, using (\ref{e.q:2.8}) and (\ref{e.q:2.9})
we arrive at%
$$
g(\left[  X,Y\right]  ,V)=g(A_{X}BY-A_{Y}BX,JV)+\frac{1}{\lambda^{2}}g^{\prime}(F_{\ast}\nabla
_{X}CY,F_{\ast}JV)-\frac{1}{\lambda^{2}}g^{\prime}(F_{\ast}\nabla
_{Y}CX,F_{\ast}JV).
$$
Thus, from (\ref{e.q:2.15}) and Lemma \ref{lem1} (i) we derive%
\begin{align*}
g(\left[  X,Y\right]  ,V)  &   =g(A_{X}BY-A_{Y}BX,JV)-g(\mathcal{H}%
\operatorname{grad}\ln\lambda,X)g(CY,JV)\\
&  -g(\mathcal{H}\operatorname{grad}\ln\lambda
,CY)g(X,JV)+g(X,CY)g(\mathcal{H}\operatorname{grad}\ln\lambda,JV)\\
&  +\frac{1}{\lambda^2}g^{\prime}(\nabla_{F_{\ast}X}F_{\ast}CY,F_{\ast
}JV)+g(\mathcal{H}\operatorname{grad}\ln\lambda
,Y)g(CX,JV)\\
&  +g(\mathcal{H}\operatorname{grad}\ln\lambda
,CX)g(Y,JV)-g(Y,CX)g(\mathcal{H}\operatorname{grad}\ln\lambda,JV)\\
&  -\frac{1}{\lambda^2}g^{\prime}(\nabla_{F_{\ast}Y}F_{\ast}CX,F_{\ast}JV).
\end{align*}
Moreover, using (\ref{e.q:3.4}), we obtain%
\begin{align*}
g(\left[  X,Y\right]  ,V)  &  =g(A_{X}BY-A_{Y}BX,JV)-g(\mathcal{H}%
\operatorname{grad}\ln\lambda,CY)g(X,JV)\\
&  +g(\mathcal{H}\operatorname{grad}%
\ln\lambda,CX)g(Y,JV)-2g(CX,Y)g(\mathcal{H}\operatorname{grad}\ln\lambda,JV)\\
&-\frac{1}%
{\lambda^{2}}g^{\prime}(\nabla_{F_{\ast}Y}F_{\ast}CX -\nabla_{F_{\ast}X}%
F_{\ast}CY,F_{\ast}JV),
\end{align*}
which proves $(a)\Leftrightarrow(b)$.
\end{proof}

From Theorem \ref{teo1}, we deduce the following which shows that a conformal anti-invariant submersion with integrable $(\ker F_{\ast})^{\perp}$ turns out to be a horizontally  homothetic submersion.

\begin{theorem}
Let $F$ be a conformal anti-invariant submersion from a K\"{a}hler manifold
$(M,g,J)$ to a Riemannian manifold $(N,g^{\prime})$. Then any two conditions
below imply the three:

$%
\begin{array}
[c]{ll}%
(i) & (\ker F_{\ast})^{\perp}\text{ is integrable}\\
(ii) & F\text{ is horizontally homotetic.}\\
(iii) & g^{\prime}(\nabla_{Y}^{F}F_{\ast}CX-\nabla_{X}^{F}F_{\ast}CY,F_{\ast
}JV)=\lambda^{2}g(A_{X}BY-A_{Y}BX,JV)
\end{array}
$

for $X,Y\in\Gamma((\ker F_{\ast})^{\perp})$ and $V\in\Gamma(\ker F_{\ast})$.
\end{theorem}

\begin{proof}
For $X,Y\in\Gamma((\ker F_{\ast})^{\perp})$ and $V\in\Gamma(\ker F_{\ast}),$
from Theorem \ref{teo1}, we have%
\begin{align*}
g(\left[  X,Y\right]  ,V)  &  =g(A_{X}BY-A_{Y}BX,JV)-g(\mathcal{H}%
\operatorname{grad}\ln\lambda,CY)g(X,JV)\\
&+g(\mathcal{H}\operatorname{grad}%
\ln\lambda,CX)g(Y,JV)-2g(CX,Y)g(\mathcal{H}\operatorname{grad}\ln\lambda,JV)\\
&-\frac{1}%
{\lambda^{2}}g^{\prime}(\nabla_{F_{\ast}Y}F_{\ast}CX-\nabla_{F_{\ast}X}%
F_{\ast}CY,F_{\ast}JV).
\end{align*}
Now, if we have $(i)$ and $(iii)$, then we arrive at%
\begin{eqnarray}
&&-g(\mathcal{H}\operatorname{grad}\ln\lambda,CY)g(X,JV)+g(\mathcal{H}%
\operatorname{grad}\ln\lambda,CX)g(Y,JV)\nonumber\\
&&-2g(CX,Y)g(\mathcal{H}%
\operatorname{grad}\ln\lambda,JV)=0. \label{e.q:3.6}%
\end{eqnarray}
Now, taking $Y=JV$ in (\ref{e.q:3.6}) for $V\in \Gamma(kerF_{\ast})$ and using (\ref{e.q:3.4}), we get%
\[
g(\mathcal{H}\operatorname{grad}\ln\lambda,CX)g(V,V)=0.
\]
Hence $\lambda$ is a constant on $\Gamma(\mu)$. On the other hand,
 taking $Y=CX$ in (\ref{e.q:3.6}) for $X\in\Gamma(\mu)$ and using (\ref{e.q:3.4}) we
derive%
\begin{align*}
-g(\mathcal{H}\operatorname{grad}\ln\lambda,C^{2}X)g(X,JV)+& g(\mathcal{H}%
\operatorname{grad}\ln\lambda,CX)g(CX,JV)\\
& -2g(CX,CX)g(\mathcal{H}%
\operatorname{grad}\ln\lambda,JV)=0,
\end{align*}
hence, we arrive at
\[
g(CX,CX)g(\mathcal{H}\operatorname{grad}\ln\lambda,JV)=0.
\]
From above equation, $\lambda$ is a constant on $\Gamma(J(\ker F_{\ast}))$.
Similarly, one can obtain the other assertions.
\end{proof}

We say that a conformal anti-invariant submersion is a conformal Lagrangian
submersion if $J(\ker F_{\ast})=(\ker F_{\ast})^{\perp}$. From Theorem \ref{teo1}, we have the following.

\begin{corollary}
Let $F:(M,g,J)\rightarrow(N,g^{\prime})$ be a conformal Lagrangian submersion,
where $(M,g,J)$ is a K\"{a}hler manifold and $(N,g^{\prime})$ is a Riemannian
manifold. Then the following assertions are equivalent to each other;

$%
\begin{array}
[c]{ll}%
(i) & (\ker F_{\ast})^{\perp}\text{ is integrable.}\\
(ii) & A_{X}JY=A_{Y}JX\\
(iii) & (\nabla F_{\ast})(Y,JX)=(\nabla F_{\ast})(X,JY)
\end{array}
$

for $X,Y\in\Gamma((\ker F_{\ast})^{\perp})$.
\end{corollary}

\begin{proof}
For $X,Y\in\Gamma((\ker F_{\ast})^{\perp})$ and $V\in\Gamma(\ker F_{\ast}),$
we see from Definition \ref{def3.1}, $JV\in\Gamma((\ker F_{\ast})^{\perp})$
and $JY\in\Gamma(J(\ker F_{\ast}))$. From Theorem \ref{teo1} we have%
\begin{align*}
g(\left[  X,Y\right]  ,V)  &  =g(A_{X}BY-A_{Y}BX,JV)-g(\mathcal{H}%
\operatorname{grad}\ln\lambda,CY)g(X,JV)\\
& +g(\mathcal{H}\operatorname{grad}%
\ln\lambda,CX)g(Y,JV) -2g(CX,Y)g(\mathcal{H}\operatorname{grad}\ln\lambda,JV)\\
& -\frac{1}%
{\lambda^{2}}g^{\prime}(\nabla_{F_{\ast}Y}F_{\ast}CX-\nabla_{F_{\ast}X}%
F_{\ast}CY,F_{\ast}JV).
\end{align*}
Since $F$ is a conformal Lagrangian submersion, we derive%
\[
g(\left[  X,Y\right]  ,V)=g(A_{X}BY-A_{Y}BX,JV)=0
\]
which shows $(i)\Leftrightarrow(ii)$. On the other hand using Definition
\ref{def3.1} and (\ref{e.q:2.8}) we arrive at%
\begin{align*}
g(A_{X}BY,JV)-g(A_{Y}BX,JV)  &  =\frac{1}{\lambda^{2}}g^{\prime
}(F_{\ast}A_{X}BY,F_{\ast}JV)-\frac{1}{\lambda^{2}}g^{\prime}(F_{\ast}%
A_{Y}BX,F_{\ast}JV)\\
& =\frac{1}{\lambda^{2}}g^{\prime}(F_{\ast}(\nabla_{X}BY),F_{\ast
}JV)-\frac{1}{\lambda^{2}}g^{\prime}(F_{\ast}(\nabla_{Y}BX),F_{\ast}JV).
\end{align*}
Now, using (\ref{e.q:2.15}) we obtain%
\begin{align*}
\frac{1}{\lambda^{2}}\{g^{\prime}(F_{\ast}(\nabla_{X}BY),F_{\ast}JV)&-g^{\prime}(F_{\ast}(\nabla_{Y}BX),F_{\ast}JV)\}\\
& =\frac{1}{\lambda^{2}}g^{\prime}(-(\nabla F_{\ast})(X,BY)+\nabla
_{F_{\ast}X}F_{\ast}BY,F_{\ast}JV)\\
& -\frac{1}{\lambda^{2}}g^{\prime}(-(\nabla F_{\ast}%
)(Y,BX)+\nabla_{F_{\ast}Y}F_{\ast}BX,F_{\ast}JV)\\
&=\frac{1}{\lambda^{2}}\{g^{\prime}((\nabla F_{\ast})(Y,BX)-(\nabla F_{\ast})(X,BY), F_{\ast}JV)
\end{align*}
which tells that $(ii)\Leftrightarrow(iii)$.
\end{proof}
For the geometry of leaves of the horizontal distribution, we have the following theorem.
\begin{theorem}
\label{teo2}Let $F$ be a conformal anti-invariant submersion from a K\"{a}hler
manifold $(M,g,J)$ to a Riemannian manifold $(N,g^{\prime})$. Then the
following assertions are equivalent to each other;

$%
\begin{array}
[c]{ll}%
(i) & (\ker F_{\ast})^{\perp}\text{ defines a totally geodesic foliation on
}M.
\end{array}
$

$%
\begin{array}
[c]{lll}%
(ii) & \frac{1}{\lambda^{2}}g^{\prime}(\nabla_{F_{\ast}X}F_{\ast}CY,F_{\ast
}JV)= & -g(A_{X}BY,JV)+g(\mathcal{H}\operatorname{grad}\ln\lambda,CY)g(X,JV)\\
&  & -g(\mathcal{H}\operatorname{grad}\ln\lambda,JV)g(X,CY)
\end{array}
$

for $X,Y\in\Gamma((\ker F_{\ast})^{\perp})$ and $V\in\Gamma(\ker F_{\ast})$.
\end{theorem}

\begin{proof}
From (\ref{e.q:2.1}), (\ref{e.q:2.2}), (\ref{e.q:2.8}), (\ref{e.q:2.9}), (\ref{e.q:3.2}) and (\ref{e.q:3.1}) we
get
\[
g(\nabla_{X}Y,V)=g(A_{X}BY,JV)+g(\mathcal{H}\nabla_{X}CY,JV).
\]
Since $F$ is a conformal submersion, using (\ref{e.q:2.15}) and Lemma \ref{lem1}
(i) we arrive at%
\begin{align*}
g(\nabla_{X}Y,V) &  =g(A_{X}BY,JV)-\frac{1}{\lambda^{2}}g(\mathcal{H}\operatorname{grad}%
\ln\lambda,X)g^{\prime}(F_{\ast}CY,F_{\ast}JV)\\
&-\frac{1}{\lambda^{2}%
}g(\mathcal{H}\operatorname{grad}\ln\lambda,CY)g^{\prime}(F_{\ast}X,F_{\ast
}JV)  +\frac{1}{\lambda^{2}}g(X,CY)g^{\prime}(F_{\ast}(\operatorname{grad}%
\ln\lambda),F_{\ast}JV)\\
&+\frac{1}{\lambda^{2}}g^{\prime}(\nabla_{F_{\ast}%
X}F_{\ast}CY,F_{\ast}JV).
\end{align*}
Moreover, using Definition \ref{def3.1} and (\ref{e.q:3.4}) we obtain%
\begin{align*}
g(\nabla_{X}Y,V)  &  =g(A_{X}BY,JV)-g(\mathcal{H}\operatorname{grad}\ln
\lambda,CY)g(X,JV)\\
& +g(\mathcal{H}\operatorname{grad}\ln\lambda,JV)g(X,CY)+\frac{1}{\lambda^{2}}g^{\prime}(\nabla_{F_{\ast}X}F_{\ast}CY,F_{\ast}JV)
\end{align*}
which proves $(i)\Leftrightarrow(ii)$.
\end{proof}

From Theorem \ref{teo2}, we  also deduce the following characterization.

\begin{theorem}
Let $F$ be a conformal anti-invariant submersion from a K\"{a}hler manifold
$(M,g,J)$ to a Riemannian manifold $(N,g^{\prime})$. Then any two conditions
below imply the three:

$%
\begin{array}
[c]{ll}%
(i) & (\ker F_{\ast})^{\perp}\text{ defines a totally geodesic foliation on
}M.\\
(ii) & F\text{ is horizontally homotetic.}\\
(iii) & g^{\prime}(\nabla_{F_{\ast}X}F_{\ast}CY,F_{\ast}JV)=-\lambda^{2}%
g(A_{X}BY,JV)
\end{array}
$

for $X,Y\in\Gamma((\ker F_{\ast})^{\perp})$ and $V\in\Gamma(\ker F_{\ast})$.
\end{theorem}

\begin{proof}
For $X,Y\in\Gamma((\ker F_{\ast})^{\perp})$ and $V\in\Gamma(\ker F_{\ast}),$
from Theorem \ref{teo2}, we have%
\begin{align*}
g(\nabla_{X}Y,V)  &  =g(A_{X}BY,JV)-g(\mathcal{H}\operatorname{grad}\ln
\lambda,CY)g(X,JV)+g(\mathcal{H}\operatorname{grad}\ln\lambda,JV)g(X,CY)\\
&  +\frac{1}{\lambda^{2}}g^{\prime}(\nabla_{F_{\ast}X}F_{\ast}CY,F_{\ast}JV).
\end{align*}
Now, if we have $(i)$ and $(iii)$, then we obtain%
\begin{equation}
-g(\mathcal{H}\operatorname{grad}\ln\lambda,CY)g(X,JV)+g(\mathcal{H}%
\operatorname{grad}\ln\lambda,JV)g(X,CY)=0. \label{e.q:3.7}%
\end{equation}
Now, taking $X=CY$ in (\ref{e.q:3.7}) and using (\ref{e.q:3.4}), we get%
\[
g(\mathcal{H}\operatorname{grad}\ln\lambda,JV)g(CY,CY)=0.
\]
Thus, $\lambda$ is a constant on $\Gamma(J(\ker F_{\ast}))$. On the other
hand, taking $X=JV$ in (\ref{e.q:3.7}) and using (\ref{e.q:3.4}) we derive
\[
g(\mathcal{H}\operatorname{grad}\ln\lambda,CY)g(V,V)=0.
\]
From above equation, $\lambda$ is a constant on $\Gamma(\mu)$. Similarly, one
can obtain the other assertions.
\end{proof}

In particular, if $F$ is a conformal Lagrangian submersion, then we
have the following.

\begin{corollary}
\label{cor1}Let $F:(M,g,J)\rightarrow(N,g^{\prime})$ be a conformal Lagrangian
submersion, where $(M,g,J)$ is a K\"{a}hler manifold and $(N,g^{\prime})$ is a
Riemannian manifold. Then the following assertions are equivalent to each other;

$%
\begin{array}
[c]{ll}%
(i) & (\ker F_{\ast})^{\perp}\text{ defines a totally geodesic foliation on
}M.\\
(ii) & A_{X}JY=0\\
(iii) & (\nabla F_{\ast})(X,JY)=0
\end{array}
$

for $X,Y\in\Gamma((\ker F_{\ast})^{\perp})$.
\end{corollary}

\begin{proof}
For $X,Y\in\Gamma((\ker F_{\ast})^{\perp})$ and $V\in\Gamma(\ker F_{\ast}),$
we see from Definition \ref{def3.1}, $JV\in\Gamma((\ker F_{\ast})^{\perp})$
and $JY\in\Gamma(J(\ker F_{\ast}))$. From Theorem \ref{teo2} we have%
\begin{align*}
g(\nabla_{X}Y,V)  &  =g(A_{X}BY,JV)-g(\mathcal{H}\operatorname{grad}\ln
\lambda,CY)g(X,JV)+g(\mathcal{H}\operatorname{grad}\ln\lambda,JV)g(X,CY)\\
&  +\frac{1}{\lambda^{2}}g^{\prime}(\nabla_{F_{\ast}X}F_{\ast}CY,F_{\ast}JV).
\end{align*}
Since $F$ is a conformal Lagrangian submersion, we derive%
\[
g(\nabla_{X}Y,V)=g(A_{X}BY,JV)
\]
which shows $(i)\Leftrightarrow(ii)$. On the other hand using (\ref{e.q:2.8}) we get%
$$
g(A_{X}BY,JV) = g(\nabla_{X}BY,JV).$$
Since $F$ is a conformal submersion, we have
$$g(A_{X}BY,JV)=\frac{1}{\lambda^{2}}g^{\prime}(F_{\ast}\nabla_{X}BY,F_{\ast
}JV).$$
Then using (\ref{e.q:2.15}) we get
$$g(A_{X}BY,JV)=- \frac{1}{\lambda^{2}}g^{\prime}((\nabla F_{\ast})(X,BY),F_{\ast
}JV)$$
which tells that $(ii)\Rightarrow(iii)$.
\end{proof}

In the sequel we are going to investigate the geometry of leaves of the  distribution $\ker F_{\ast}$.

\begin{theorem}
\label{teo3}Let $F:(M,g,J)\rightarrow(N,g^{\prime})$ be a conformal anti-invariant
submersion, where $(M,g,J)$ is a K\"{a}hler manifold and $(N,g^{\prime})$ is a
Riemannian manifold. Then the following assertions are equivalent to each other;

$%
\begin{array}
[c]{ll}%
(i) & \ker F_{\ast}\text{ defines a totally geodesic foliation on }M.
\end{array}
$

$%
\begin{array}
[c]{lll}%
(ii) & -\frac{1}{\lambda^{2}}g^{\prime}(\nabla_{F_{\ast}JW}F_{\ast}JV,F_{\ast
}JCX)= & g(T_{V}JW,BX)+g(JW,JV)g(\mathcal{H}\operatorname{grad}\ln\lambda,JCX)
\end{array}
$

for $V,W\in\Gamma(\ker F_{\ast})$ and $X\in\Gamma((\ker F_{\ast})^{\perp})$.
\end{theorem}

\begin{proof}
For $V,W\in\Gamma(\ker F_{\ast})$ and $X\in\Gamma((\ker F_{\ast})^{\perp})$,
from (\ref{e.q:2.1}), (\ref{e.q:2.2}), (\ref{e.q:2.7}) and (\ref{e.q:3.2}) we get%
$$
g(\nabla_{V}W,X) =g(T_{V}JW,BX)+g(\mathcal{H}\nabla_{V}JW,CX).$$
Since $\nabla$ is torsion free and $\left[  V,JW\right]\in \Gamma(\ker F_{\ast})$, we obtain
$$g(\nabla_{V}W,X)=g(T_{V}JW,BX)+g(\nabla_{JW}V,CX).$$
Using  (\ref{e.q:2.2}) and (\ref{e.q:2.9}) we have%
\begin{align*}
 g(\nabla_{V}W,X)=g(T_{V}JW,BX)+g(\nabla_{JW}JV,JCX),
\end{align*}
here we have used that $\mu$ is invariant. Since $F$ is a conformal submersion, using (\ref{e.q:2.15}) and Lemma \ref{lem1}
(i) we obtain%
\begin{align*}
g(\nabla_{V}W,X)  &  =g(T_{V}JW,BX)-\frac{1}{\lambda^{2}}g(\mathcal{H}\operatorname{grad}%
\ln\lambda,JW)g^{\prime}(F_{\ast}JV,F_{\ast}JCX)\\
&-\frac{1}{\lambda^{2}%
}g(\mathcal{H}\operatorname{grad}\ln\lambda,JV)g^{\prime}(F_{\ast}JW,F_{\ast
}JCX)  +g(JW,JV)\frac{1}{\lambda^{2}}g^{\prime}(F_{\ast}\operatorname{grad}%
\ln\lambda,F_{\ast}JCX)\nonumber\\
& +\frac{1}{\lambda^{2}}g^{\prime}(\nabla_{F_{\ast}%
JW}F_{\ast}JV,F_{\ast}JCX).
\end{align*}
Moreover, using Definition \ref{def3.1} and (\ref{e.q:3.4}) we derive%
\begin{align*}
g(\nabla_{V}W,X)& =g(T_{V}JW,BX)+g(JW,JV)g(\mathcal{H}\operatorname{grad}%
\ln\lambda,JCX)\\
& +\frac{1}{\lambda^{2}}g^{\prime}(\nabla_{F_{\ast}JW}F_{\ast
}JV,F_{\ast}JCX)
\end{align*}
which proves $(i)\Leftrightarrow(ii)$.
\end{proof}

From Theorem \ref{teo3}, we deduce have the following result.

\begin{theorem}
Let $F$ be a conformal anti-invariant submersion from a K\"{a}hler manifold
$(M,g,J)$ to a Riemannian manifold $(N,g^{\prime})$. Then any two conditions
below imply the three:

$%
\begin{array}
[c]{ll}%
(i) & \ker F_{\ast}\text{ defines a totally geodesic foliation on
}M.\\
(ii) & \lambda\text{ is a constant on }\Gamma(\mu)\text{.}\\
(iii) & \frac{1}{\lambda^{2}}g^{\prime}(\nabla_{F_{\ast}JW}F_{\ast}JV,F_{\ast
}JCX)=-g(T_{V}JW,JX)
\end{array}
$

for $V,W\in\Gamma(\ker F_{\ast})$ and $X\in\Gamma((\ker F_{\ast})^{\perp})$.
\end{theorem}

\begin{proof}
For $V,W\in\Gamma(\ker F_{\ast})$ and $X\in\Gamma((\ker F_{\ast})^{\perp}),$
from Theorem \ref{teo3}, we have%
\[
g(\nabla_{V}W,X)=g(T_{V}JW,BX)+g(JW,JV)g(\mathcal{H}\operatorname{grad}%
\ln\lambda,JCX)+\frac{1}{\lambda^{2}}g^{\prime}(\nabla_{F_{\ast}JW}F_{\ast
}JV,F_{\ast}JCX).
\]
Now, if we have $(i)$ and $(iii)$, then we get%
\[
g(JW,JV)g(\mathcal{H}\operatorname{grad}\ln\lambda,JCX)=0.
\]
From above equation, $\lambda$ is a constant on $\Gamma(\mu)$. Similarly, one
can obtain the other assertions.
\end{proof}

If $F$ is a conformal Lagrangian submersion, then (\ref{e.q:3.3}) implies that
$TN=F_{\ast}(J(\ker F_{\ast}))$. Hence we have the following.

\begin{corollary}
\label{cor2}Let $F:(M,g,J)\rightarrow(N,g^{\prime})$ be a conformal Lagrangian
submersion, where $(M,g,J)$ is a K\"{a}hler manifold and $(N,g^{\prime})$ is a
Riemannian manifold. Then the following assertions are equivalent to each other;

$%
\begin{array}
[c]{ll}%
(i) & \ker F_{\ast}\text{ defines a totally geodesic foliation on
}M.\\
(ii) & T_{V}JW=0
\end{array}
$

for $V,W\in\Gamma(\ker F_{\ast})$.
\end{corollary}

\begin{proof}
For $V,W\in\Gamma(\ker F_{\ast})$ and $X\in\Gamma((\ker F_{\ast})^{\perp}),$
from Theorem \ref{teo3} we have%
\[
g(\nabla_{V}W,X)=g(T_{V}JW,BX)+g(JW,JV)g(\mathcal{H}\operatorname{grad}%
\ln\lambda,JCX)+\frac{1}{\lambda^{2}}g^{\prime}(\nabla_{F_{\ast}JW}F_{\ast
}JV,F_{\ast}JCX).
\]
Since $F$ is a conformal Lagrangian submersion, we get%
\[
g(\nabla_{V}W,X)=g(T_{V}JW,BX)
\]
which shows $(i)\Leftrightarrow(ii)$.
\end{proof}
\section{{\protect \textbf{Harmonicity of Conformal Anti-invariant Submersions}\label%
{sect-prel}}}
In this section, we are going to find necessary and sufficient conditions for a conformal anti-invariant submersions to be harmonic. We also investigate the necessary and sufficient conditions for such submersions to be totally geodesic.
\begin{theorem}
\label{teo4}Let $F:(M^{2m+2r},g,J)\rightarrow(N^{m+2r},g^{\prime})$ be a
conformal anti-invariant submersion, where $(M,g,J)$ is a K\"{a}hler manifold
and $(N,g^{\prime})$ is a Riemannian manifold. Then the tension field $\tau$
of $F$ is%
\begin{equation}%
\begin{array}
[c]{ll}%
\tau(F)= & -\frac{1}{m}F_{\ast}(\mu^{\ker F_{\ast}})+(\frac{2}{\lambda^{2}%
}-(m+2r))F_{\ast}(\operatorname{grad}\ln\lambda)\mid_{F_{\ast}(JV)}\\
& +(\frac{2}{\lambda^{2}}-(m+2r))F_{\ast}(\operatorname{grad}\ln\lambda
)\mid_{F_{\ast}(\mu)}%
\end{array}
\label{e.q:3.8}%
\end{equation}
where $\mu^{\ker F_{\ast}}$ is the mean curvature vector field of the
distribution of $\ker F_{\ast}$.
\end{theorem}

\begin{proof}
Let $\{e_{1},...,e_{m},Je_{1},...,Je_{m},\mu_{1},...,\mu_{r},J\mu_{r}%
,...,J\mu_{r}\}$ be an orthonormal basis of $\Gamma(TM)$ such that
$\{e_{1},...,e_{m}\}$ is orthonormal basis of $\Gamma(\ker F_{\ast})$,
$\{Je_{1},...,Je_{m}\}$ is orthonormal basis of $\Gamma(J\ker F_{\ast})$ and
$\{\mu_{1},...,\mu_{r},J\mu_{r},...,J\mu_{r}\}$ is orthonormal basis of
$\Gamma(\mu)$. Then the trace of second fundamental form (restriction to $\ker
F_{\ast}\times\ker F_{\ast}$) is given by%
\[
trace^{\ker F_{\ast}}\nabla F_{\ast}=\sum_{i=1}^{m}(\nabla F_{\ast}%
)(e_{i},e_{i}).
\]
Then using (\ref{e.q:2.15}) we obtain%
\begin{align}
trace^{\ker F_{\ast}}\nabla F_{\ast}  &  =-\frac{1}{m}F_{\ast}(\mu^{\ker F_{\ast}}).\label{e.q:3.9}
\end{align}
In a similar way, we have%
\[
trace^{(\ker F_{\ast})^{\perp}}\nabla F_{\ast}=\sum_{i=1}^{m}(\nabla F_{\ast
})(Je_{i},Je_{i})+\sum_{i=1}^{2r}(\nabla F_{\ast})(\mu_{i},\mu_{i}).
\]
Using Lemma \ref{lem1} (i) we arrive at%
\begin{align*}
trace^{(\ker F_{\ast})^{\perp}}\nabla F_{\ast}  &  =\sum_{i=1}^{m}2g(\operatorname{grad}\ln\lambda,Je_{i})F_{\ast}%
(Je_{i})-mF_{\ast}(\operatorname{grad}\ln\lambda)\\
&  +\sum_{i=1}^{2r}2g(\operatorname{grad}\ln\lambda,\mu_{i})F_{\ast}(\mu
_{i})-2rF_{\ast}(\operatorname{grad}\ln\lambda).
\end{align*}
Since $F$ is a conformal anti-invariant submersion, we derive%
\begin{align}
trace^{(\ker F_{\ast})^{\perp}}\nabla F_{\ast}  &  =\sum_{i=1}^{m}2\frac
{1}{\lambda^{2}}g^{\prime}(F_{\ast}(\operatorname{grad}\ln\lambda),F_{\ast
}(Je_{i}))F_{\ast}(Je_{i})-mF_{\ast}(\operatorname{grad}\ln\lambda)\nonumber\\
&  +\sum_{i=1}^{2r}2\frac{1}{\lambda^{2}}g^{\prime}(F_{\ast}%
(\operatorname{grad}\ln\lambda),F_{\ast}(\mu_{i}))F_{\ast}(\mu_{i})-2rF_{\ast
}(\operatorname{grad}\ln\lambda)\label{e.q:3.10}\\
&  =(\frac{2}{\lambda^{2}}-(m+2r))F_{\ast}(\operatorname{grad}\ln\lambda
)\mid_{F_{\ast}(JV)}+(\frac{2}{\lambda^{2}}\nonumber\\
& -(m+2r))F_{\ast}%
(\operatorname{grad}\ln\lambda)\mid_{F_{\ast}(\mu)}.\nonumber
\end{align}
Then proof follows from (\ref{e.q:3.9}) and (\ref{e.q:3.10}).
\end{proof}

From Theorem \ref{teo4} we deduce that:

\begin{theorem}
Let $F:(M^{2m+2r},g,J)\rightarrow(N^{m+2r},g^{\prime})$ be a conformal
anti-invariant submersion such that $\frac{2}{(m+2r)}\neq\lambda^{2}$, where $(M,g,J)$ is a K\"{a}hler manifold and
$(N,g^{\prime})$ is a Riemannian manifold. Then any three conditions below
imply the fourth:

$%
\begin{array}
[c]{ll}%
(i) & F\text{ is harmonic}\\
(ii) & \text{The fibres are minimal}\\
(iii) & \lambda\text{ is a constant on }\Gamma(J\ker F_{\ast})\\
(iv) & \lambda\text{ is a constant on }\Gamma(\mu).
\end{array}
$
\end{theorem}

\begin{proof}
From (\ref{e.q:3.8}), we have%
\[%
\begin{array}
[c]{ll}%
\tau(F)= & -\frac{1}{m}F_{\ast}(\mu^{\ker F_{\ast}})+(\frac{2}{\lambda^{2}%
}-(m+2r))F_{\ast}(\operatorname{grad}\ln\lambda)\mid_{F_{\ast}(JV)}\\
& +(\frac{2}{\lambda^{2}}-(m+2r))F_{\ast}(\operatorname{grad}\ln\lambda
)\mid_{F_{\ast}(\mu)}.
\end{array}
\]
Now, if we have $(i),(ii)$ and $(iii)$ then $\lambda$ is a constant on
$\Gamma(\mu)$.
\end{proof}
We also have the following result.
\begin{corollary}Let $F$ be a conformal anti-invariant submersion from a K\"{a}hler manifold
$(M,g,J)$ to a Riemannian manifold $(N,g^{\prime})$. If $\frac{2}{(m+2r)}=\lambda^{2}$ then $F$ is harmonic if and only if the fibres  are minimal.
\end{corollary}
Now we obtain necessary and sufficient condition for conformal anti-invariant
submersion to be totally geodesic. We recall that a differentiable map $F$
between two Riemannian manifolds is called totally geodesic if
\[
(\nabla F_{\ast})(X,Y)=0,\mbox{ }\text{for all }X,Y\in\Gamma(TM).
\]
A geometric interpretation of a totally
geodesic map is that it maps every geodesic in the total space into a geodesic
in the base space in proportion to arc lengths.

\begin{theorem}
Let $F$ be a conformal anti-invariant submersion from a K\"{a}hler manifold
$(M,g,J)$ to a Riemannian manifold $(N,g^{\prime})$. Then $F$ is a totally
geodesic map if and only if%
\begin{align}
-\nabla_{X}^{F}F_{\ast}Y& =F_{\ast}(J(A_{X}JY_{1}+\mathcal{V}\nabla_{X}%
BY_{2}+A_{X}CY_{2})+C(\mathcal{H}\nabla_{X}JY_{1}\nonumber\\
& +A_{X}BY_{2}+\mathcal{H}%
\nabla_{X}CY_{2})) \label{e.q:3.11}%
\end{align}
for any $X,Y=Y_{1}+Y_{2}\in\Gamma(TM),$ where $Y_{1}%
\in\Gamma(\ker F_{\ast})$ and $Y_{2}$ $\in\Gamma((\ker F_{\ast})^{\perp})%
$.
\end{theorem}

\begin{proof}
Using (\ref{e.q:2.2}) and (\ref{e.q:2.15}) we have
\begin{align*}
(\nabla F_{\ast})(X,Y)  &  =\nabla_{X}^{F}F_{\ast}Y+F_{\ast}(J\nabla_{X}JY)
\end{align*}
for any $X, Y \in\Gamma(TM)$.
Then from (\ref{e.q:2.8}) and (\ref{e.q:3.2}) we get
\begin{align*}
(\nabla F_{\ast})(X,Y)&  =\nabla_{X}^{F}F_{\ast}Y+F_{\ast}(JA_{X}JY_{1}+B\mathcal{H}\nabla_{X}%
JY_{1}+C\mathcal{H}\nabla_{X}JY_{1}+BA_{X}BY_{2}\\
&  +CA_{X}BY_{2}+J\mathcal{V}\nabla_{X}BY_{2}+JA_{X}CY_{2}+B\mathcal{H}%
\nabla_{X}CY_{2}+C\mathcal{H}\nabla_{X}CY_{2})
\end{align*}
for any $Y=Y_{1}+Y_{2}\in\Gamma(TM)$, where $Y_{1}%
\in\Gamma(\ker F_{\ast})$ and $Y_{2}$ $\in\Gamma((\ker F_{\ast})^{\perp})%
$.
Thus taking into account the vertical parts, we find
\begin{align*}
(\nabla F_{\ast})(X,Y)&  =\nabla_{X}^{F}F_{\ast}Y+F_{\ast}(J(A_{X}JY_{1}+\mathcal{V}\nabla_{X}%
BY_{2}+A_{X}CY_{2})+C(\mathcal{H}\nabla_{X}JY_{1}\\
+& A_{X}BY_{2}+\mathcal{H}%
\nabla_{X}CY_{2})).
\end{align*}
Thus $(\nabla F_{\ast})(X,Y)=0$ if and only if  the equation (\ref{e.q:3.11}) is satisfied.
\end{proof}
We now present the following definition.
\begin{definition}
Let $F$ be a conformal anti-invariant submersion from a K\"{a}hler manifold
$(M,g,J)$ to a Riemannian manifold $(N,g^{\prime})$. Then $F$ is called a $(J\ker
F_{\ast},\mu)$-totally geodesic map if%
\[
(\nabla F_{\ast})(JU,X)=0,\mbox{ }\text{for }U\in\Gamma(\ker F_{\ast})\text{
and }X\in\Gamma((\ker F_{\ast})^{\perp}).
\]

\end{definition}
In the sequel we show that this notion has an important effect on the character of the conformal submersion.
\begin{theorem}
Let $F$ be a conformal anti-invariant submersion from a K\"{a}hler manifold
$(M,g,J)$ to a Riemannian manifold $(N,g^{\prime})$. Then $F$ is a $(J\ker F_{\ast
},\mu)$-totally geodesic map if and only if $F$ is horizontally homotetic map.
\end{theorem}

\begin{proof}
For $U\in\Gamma(\ker F_{\ast})$ and $X\in\Gamma(\mu),$ from Lemma \ref{lem1} (i),
we have%
\[
(\nabla F_{\ast})(JU,X)=JU(\ln\lambda)F_{\ast}(X)+X(\ln\lambda)F_{\ast
}(JU)-g(JU,X)F_{\ast}(\operatorname{grad}\ln\lambda).
\]
From above equation, if $F$ is a horizontally homotetic then $(\nabla F_{\ast
})(JU,X)=0$. Conversely, if $(\nabla F_{\ast})(JU,X)=0$, we obtain%
\begin{equation}
JU(\ln\lambda)F_{\ast}(X)+X(\ln\lambda)F_{\ast}(JU)=0. \label{e.q:3.12}%
\end{equation}
Taking inner product in (\ref{e.q:3.12}) with $F_{\ast}(JU)$ and since $F$ is
a conformal submersion, we write
\[
g(\operatorname{grad}\ln\lambda,JU)g^{\prime}(F_{\ast}X,F_{\ast}%
JU)+g(\operatorname{grad}\ln\lambda,X)g^{\prime}(F_{\ast}JU,F_{\ast}JU)=0.
\]
Above equation implies that  $\lambda$ is a constant on $\Gamma(\mu)$. On the other
hand, taking inner product in (\ref{e.q:3.12}) with $F_{\ast}X$, we have%
\[
g(\operatorname{grad}\ln\lambda,JU)g^{\prime}(F_{\ast}X,F_{\ast}%
X)+g(\operatorname{grad}\ln\lambda,X)g^{\prime}(F_{\ast}JU,F_{\ast}X)=0.
\]
From above equation, it follows that $\lambda$ is a constant on $\Gamma(J(\ker F_{\ast}))$.
Thus $\lambda$ is a constant on $\Gamma((\ker F_{\ast})^{\perp})$. Hence proof
is complete.
\end{proof}
Here we present another result on conformal anti-invariant submersion to be totally geodesic.
\begin{theorem}
Let $F$ be a conformal anti-invariant submersion from a K\"{a}hler manifold
$(M,g,J)$ to a Riemannian manifold $(N,g^{\prime})$. Then $F$ is a totally geodesic
map if and only if

$%
\begin{array}
[c]{ll}%
(i) & T_{U}JV=0 \text{ and }\mathcal{H}\nabla_{U}JV\in\Gamma(J\ker F_{\ast}),\\
(ii) & F\text{ is horizontally homotetic map,}\\
(iii) & \hat{\nabla}_{V}BX+T_{V}CX=0\\
& T_{V}BX+\mathcal{H}\nabla_{V}CX\in\Gamma(J\ker F_{\ast})
\end{array}
$

for $X,Y\in\Gamma((\ker F_{\ast})^{\perp})$ and $U,V\in\Gamma(\ker F_{\ast})$.
\end{theorem}

\begin{proof}
For any $U,V\in\Gamma(\ker F_{\ast}),$ from (\ref{e.q:2.2}) and (\ref{e.q:2.15})
we have%
\begin{align*}
(\nabla F_{\ast})(U,V)  &  =F_{\ast}(J\nabla_{U}JV).
\end{align*}
Then (\ref{e.q:3.2}) and (\ref{e.q:2.7}) implies that
\begin{align*}
(\nabla F_{\ast})(U,V)&  =F_{\ast}(JT_{U}JV+C\mathcal{H}\nabla_{U}JV).
\end{align*}
From above equation,  $(\nabla F_{\ast})(U,V)=0$  if and only if
\begin{equation}
F_{\ast}(JT_{U}JV+C\mathcal{H}\nabla_{U}JV)=0. \label{e.q:3.13}%
\end{equation}
This implies $T_{U}JV=0$ and $\mathcal{H}\nabla_{U}JV \in \Gamma(J\ker F_{\ast})$.
On the other hand,  from Lemma \ref{lem1} (i) we derive%
\[
(\nabla F_{\ast})(X,Y)=X(\ln\lambda)F_{\ast}(Y)+Y(\ln\lambda)F_{\ast
}(X)-g(X,Y)F_{\ast}(\operatorname{grad}\ln\lambda)
\]
for any $X,Y\in\Gamma(\mu)$. It is obvious that if $F$ is horizontally homothetic, it follows that $(\nabla F_{\ast})(X,Y)=0$. Conversely, if $(\nabla F_{\ast})(X,Y)=0$,  taking $Y=JX$ in above equation, we get%
\begin{equation}
X(\ln\lambda)F_{\ast}(JX)+JX(\ln\lambda)F_{\ast}(X)=0. \label{e.q:3.16}%
\end{equation}
Taking inner product in (\ref{e.q:3.16}) with $F_{\ast}JX$, we obtain%
\begin{equation}
g(\operatorname{grad}\ln\lambda,X)\lambda^{2}g(JX,JX)+g(\operatorname{grad}%
\ln\lambda,JX)\lambda^{2}g(X,JX)=0. \label{e.q:3.17}%
\end{equation}
From (\ref{e.q:3.17}), $\lambda$ is a constant on $\Gamma(\mu)$. On the other
hand, for $U,V\in\Gamma(\ker F_{\ast}),$ from Lemma \ref{lem1} (i) we have
\[
(\nabla F_{\ast})(JU,JV)=JU(\ln\lambda)F_{\ast}(JV)+JV(\ln\lambda)F_{\ast
}(JU)-g(JU,JV)F_{\ast}(\operatorname{grad}\ln\lambda).
\]
Again if $F$ is horizontally homothetic, then $(\nabla F_{\ast})(JU,JV)=0$.  Conversely, if $(\nabla F_{\ast})(JU,JV)=0$, putting $U$ instead of $V$ in above equation, we derive%
\begin{equation}
2JU(\ln\lambda)F_{\ast}(JU)-g(JU,JU)F_{\ast}(\operatorname{grad}\ln\lambda)=0.
\label{e.q:3.18}%
\end{equation}
Taking inner product in (\ref{e.q:3.18}) with $F_{\ast}JU$ and since $F$ is a
conformal submersion, we have%
\begin{align*}
g(JU,JU)\lambda
^{2}g(\operatorname{grad}\ln\lambda,JU)  &  =0.
\end{align*}
From above equation, $\lambda$ is a constant on $\Gamma(J\ker F_{\ast})$. Thus
$\lambda$ is a constant on $\Gamma((\ker F_{\ast})^{\perp})$. Now, for $X\in\Gamma((\ker F_{\ast})^{\perp})$ and $V\in\Gamma(\ker
F_{\ast}),$ from (\ref{e.q:2.2}) and (\ref{e.q:2.15}) we
get
\begin{align*}
(\nabla F_{\ast})(X,V)  &  =F_{\ast}(J\nabla_{V}JX).
\end{align*}
Using (\ref{e.q:3.2}) and  (\ref{e.q:2.7}) we have
\begin{align*}
(\nabla F_{\ast})(X,V) & =F_{\ast}(CT_{V}BX+J\hat{\nabla}_{V}BX+C\mathcal{H}\nabla_{V}CX+JT_{V}CX).
\end{align*}
Thus $(\nabla F_{\ast})(X,V)=0$  if and only if
\begin{align*}
F_{\ast}(CT_{V}BX+J\hat{\nabla}_{V}BX+C\mathcal{H}\nabla_{V}CX+JT_{V}CX)=0.
\end{align*}
Thus proof is complete.
\end{proof}

\section{{\protect \textbf{Decomposition Theorems}\label%
{sect-prel}}}

In this section, we obtain decomposition theorems by using the existence of
conformal anti-invariant submersions. First, we recall the following results
from \cite{PR}. Let $g$ be a Riemannian metric tensor on the manifold $B=M\times N$
and assume that the canonical foliations $D_{M}$ and $D_{N}$ intersect
perpendicularly everywhere. Then $g$ is the metric tensor of

(i) a twisted product $M\times_{f}N$ if and only if $D_{M\text{ }}$ is a
totally geodesic foliation and $D_{N}$ is a totally umbilic foliation,

(ii) a warped product $M\times_{f}N$ if and only if $D_{M\text{ }}$ is a
totally geodesic foliation and $D_{N}$ is a spheric foliation, i.e., it is
umbilic and its mean curvature vector field is parallel.

(iii) a usual product of Riemannian manifolds if and only if $D_{M}$ and
$D_{N}$ are totally geodesic foliations.

Our first decomposition theorem for a conformal anti-invariant submersion comes
from Theorem \ref{teo2} and Theorem \ref{teo3} in terms of the second
fundamental forms of such submersions.

\begin{theorem}
Let $F$ be a conformal anti-invariant submersion from a K\"{a}hler manifold
$(M,g,J)$ to a Riemannian manifold $(N,g^{\prime})$. Then $M$ is a locally
product manifold if and only if%
\begin{align*}
\frac{1}{\lambda^{2}}g^{\prime}(\nabla_{F_{\ast}X}F_{\ast}CY,F_{\ast}JV)  &
=-g(A_{X}BY,JV)+g(\mathcal{H}\operatorname{grad}\ln\lambda,CY)g(X,JV)\\
&  -g(\mathcal{H}\operatorname{grad}\ln\lambda,JV)g(X,CY)
\end{align*}
and%
\[
-\frac{1}{\lambda^{2}}g^{\prime}(\nabla_{F_{\ast}JW}F_{\ast}JV,F_{\ast
}JCX)=g(T_{V}JW,BX)+g(JW,JV)g(\mathcal{H}\operatorname{grad}\ln\lambda,JCX)
\]

for $V,W\in\Gamma(\ker F_{\ast})$ and $X,Y\in\Gamma((\ker F_{\ast})^{\perp})$.
\end{theorem}

From Corollary \ref{cor1} and Corollary \ref{cor2}, we have the following theorem.

\begin{theorem}
Let $F$ be a conformal Lagrangian submersion from a K\"{a}hler manifold
$(M,g,J)$ to a Riemannian manifold $(N,g^{\prime})$. Then $M$ is a locally
product manifold if and only if $A_{X}JY=0$ and $T_{V}JW=0$ for $X,Y\in
\Gamma((\ker F_{\ast})^{\perp})$ and $V,W\in\Gamma(\ker F_{\ast})$.
\end{theorem}

Next we obtain a decomposition theorem which is related to the notion of
twisted product manifold.

\begin{theorem}
Let $F$ be a conformal anti-invariant submersion from a K\"{a}hler manifold
$(M,g,J)$ to a Riemannian manifold $(N,g^{\prime})$. Then $M$ is a locally
twisted product manifold of the form $M_{(\ker F_{\ast})}\times_{\lambda
}M_{(\ker F_{\ast})^{\perp}}$ if and only if
\begin{equation}
-\frac{1}{\lambda^{2}}g^{\prime}(\nabla_{F_{\ast}JW}F_{\ast}JV,F_{\ast
}JCX)=g(T_{V}JW,BX)+g(JW,JV)g(\mathcal{H}\operatorname{grad}\ln\lambda,JCX)
\label{e.q:4.1}%
\end{equation}
and%
\begin{equation}
g(X,Y)H=-BA_{X}BY+CY(\ln\lambda)BX-B\mathcal{H}\operatorname{grad}\ln\lambda
g(X,CY)-JF^{\ast}(\nabla_{F_{\ast}X}F_{\ast CY}) \label{e.q:4.2}%
\end{equation}
for $V,W\in\Gamma(\ker F_{\ast})$ and $X,Y\in\Gamma((\ker F_{\ast})^{\perp}),$
where $M_{(\ker F_{\ast})^{\perp}}$ and $M_{(\ker F_{\ast})}$ are integral
manifolds of the distributions $(\ker F_{\ast})^{\perp}$ and $(\ker F_{\ast})$
and $H$ is the mean curvature vector field of $M_{(\ker F_{\ast})^{\perp}}$.
\end{theorem}
\begin{proof}
For $V,W\in\Gamma(\ker F_{\ast})$ and $X\in\Gamma((\ker F_{\ast})^{\perp})$,
from (\ref{e.q:2.1}), (\ref{e.q:2.2}), (\ref{e.q:2.7}) and (\ref{e.q:3.2}) we have%
\begin{align*}
g(\nabla_{V}W,X)  &  =g(T_{V}JW,BX)+g(\mathcal{H}\nabla_{V}JW,CX).
\end{align*}
Since $\nabla$ is torsion free and $[V,JW]\in \Gamma(\ker F_{\ast})$, we obtain
\begin{align*}
g(\nabla_{V}W,X)&  =g(T_{V}JW,BX)+g(\nabla_{JW}V,CX).
\end{align*}
Using (\ref{e.q:2.2}) and (\ref{e.q:2.9}) we get%
\begin{align*}
g(\nabla_{V}W,X)&  =g(T_{V}JW,BX)+g(\nabla_{JW}JV,JCX).
\end{align*}
Since $F$ is a conformal submersion, using (\ref{e.q:2.15}) and Lemma \ref{lem1}
(i) we arrive at%
\begin{align*}
g(\nabla_{V}W,X)  &  =g(T_{V}JW,BX)-\frac{1}{\lambda^{2}}g(\mathcal{H}\operatorname{grad}%
\ln\lambda,JW)g^{\prime}(F_{\ast}JV,F_{\ast}JCX)\\
&  -\frac{1}{\lambda^{2}%
}g(\mathcal{H}\operatorname{grad}\ln\lambda,JV)g^{\prime}(F_{\ast}JW,F_{\ast
}JCX)\\
& +g(JW,JV)\frac{1}{\lambda^{2}}g^{\prime}(F_{\ast}\operatorname{grad}%
\ln\lambda,F_{\ast}JCX)+\frac{1}{\lambda^{2}}g^{\prime}(\nabla_{F_{\ast}%
JW}F_{\ast}JV,F_{\ast}JCX).
\end{align*}
Moreover, using Definition \ref{def3.1} and (\ref{e.q:3.4}) we conclude that%
\begin{align*}
g(\nabla_{V}W,X)& =g(T_{V}JW,BX)+g(JW,JV)g(\mathcal{H}\operatorname{grad}%
\ln\lambda,JCX)\\
& +\frac{1}{\lambda^{2}}g^{\prime}(\nabla_{F_{\ast}JW}F_{\ast
}JV,F_{\ast}JCX).
\end{align*}
Thus it follows that  $M_{(\ker F_{\ast})}$ is totally geodesic if and only if the equation (\ref{e.q:4.1}) is satisfied. On
the other hand, for $V,W\in\Gamma(\ker F_{\ast})$ and $X\in\Gamma((\ker
F_{\ast})^{\perp}),$ from (\ref{e.q:2.1}), (\ref{e.q:2.2}), (\ref{e.q:2.8}), (\ref{e.q:2.9})
and (\ref{e.q:3.2}) we obtain
\begin{align*}
g(\nabla_{X}Y,V)  &  =g(A_{X}BY+\mathcal{V}\nabla_{X}BY,JV)+g(A_{X}CY+\mathcal{H}\nabla
_{X}CY,JV).
\end{align*}
Thus from (\ref{e.q:3.1}) we get%
\[
g(\nabla_{X}Y,V)=g(A_{X}BY,JV)+g(\mathcal{H}\nabla_{X}CY,JV).
\]
Since $F$ is a conformal submersion, using (\ref{e.q:2.15}) and Lemma \ref{lem1}
(i) we arrive at%
\begin{align*}
g(\nabla_{X}Y,V)  &  =g(A_{X}BY,JV)-\frac{1}{\lambda^{2}}g(\mathcal{H}\operatorname{grad}%
\ln\lambda,X)g^{\prime}(F_{\ast}CY,F_{\ast}JV)\\
& -\frac{1}{\lambda^{2}%
}g(\mathcal{H}\operatorname{grad}\ln\lambda,CY)g^{\prime}(F_{\ast}X,F_{\ast
}JV) +\frac{1}{\lambda^{2}}g(X,CY)g^{\prime}(F_{\ast}(\operatorname{grad}%
\ln\lambda),F_{\ast}JV)\\
& +\frac{1}{\lambda^{2}}g^{\prime}(\nabla_{F_{\ast}%
X}F_{\ast}CY,F_{\ast}JV).
\end{align*}
Moreover, using Definition \ref{def3.1} and (\ref{e.q:3.4}) we derive%
\begin{align*}
g(\nabla_{X}Y,V)  &  =g(A_{X}BY,JV)-g(\mathcal{H}\operatorname{grad}\ln
\lambda,CY)g(X,JV)\\
&  +g(\mathcal{H}\operatorname{grad}\ln\lambda,JV)g(X,CY)+\frac{1}{\lambda^{2}}g^{\prime}(\nabla_{F_{\ast}X}F_{\ast}CY,F_{\ast}JV).
\end{align*}
Using (\ref{e.q:2.2}) we conclude that $M_{(\ker F_{\ast})^{\perp}}$ is totally umbilical if and only if the equation (\ref{e.q:4.2}) is satisfied.
\end{proof}

However, in the sequel, we show that the
notion of conformal anti-invariant submersion puts some restrictions on the
total space for locally warped product manifold.

\begin{theorem}
Let $F$ be a conformal anti-invariant submersion from a K\"{a}hler manifold
$(M,g,J)$ to a Riemannian manifold $(N,g^{\prime})$ with $rank (\ker F_{\ast})>1$. If $M$ is a locally
warped product manifold of the form $M_{(\ker F_{\ast})^{\perp}}%
\times_{\lambda}M_{(\ker F_{\ast})}$, then either $F$ is horizontally homothetic submersion or the fibers are one dimensional.
\end{theorem}

\begin{proof}
For $V,W\in\Gamma(\ker F_{\ast})$ and $X\in\Gamma((\ker F_{\ast})^{\perp})$,
from (\ref{e.q:2.2}) and (\ref{e.q:2.6}) we get%
\begin{align*}
-X(\ln\lambda)g(U,V)  &  =JV(\ln\lambda)g(U,JX).
\end{align*}
For $X\in\Gamma(\mu)$, we derive%
\[
-X(\ln\lambda)g(U,V)=0.
\]
From above equation, we conclude that  $\lambda$ is a constant on $\Gamma(\mu)$. For
$X=JU\in\Gamma(J(\ker F_{\ast}))$ we obtain%
\begin{align}
JU(\ln\lambda)g(U,V)  &  =JV(\ln\lambda)g(U,U). \label{e.q:4.3}%
\end{align}
Interchanging the roles of  $V$ and $U$ in (\ref{e.q:4.3}) we arrive at%
\begin{equation}
JV(\ln\lambda)g(U,V)=JU(\ln\lambda)g(V,V). \label{e.q:4.4}%
\end{equation}
From (\ref{e.q:4.3}) and (\ref{e.q:4.4}) we get%
\begin{equation}
JU(\ln\lambda)=JU(\ln\lambda)\frac{g(U,V)^{2}}{\parallel U\parallel
^{2}\parallel V\parallel ^{2}}. \label{e.q:4.6}%
\end{equation}
From (\ref{e.q:4.6}), either $\lambda$ is a constant on $\Gamma(J\ker F_{\ast
})$ or $\Gamma(J\ker F_{\ast})$ is 1-dimensional. Thus proof is complete.
\end{proof}

\begin{remark} In fact, the result implies that there  are no
conformal anti-invariant submersions from  K\"{a}hler manifold
$(M,g,J)$ the form $M_{(\ker F_{\ast})^{\perp}}%
\times_{\lambda}M_{(\ker F_{\ast})}$ to a Riemannian manifold under
certain conditions.
\end{remark}
\section{{\protect \textbf{Curvature Relations for Conformal Anti-Invariant Submersions}\label%
{sect-prel}}}

In this section, we investigate sectional curvatures of the total space, the
base space and the fibres of a conformal anti-invariant submersion. Let $F$ be a
conformal anti-invariant submersion between K\"{a}hler manifold $M$ and
Riemannian manifold $N$. We denote Riemannian curvature tensors of $M$, $N$ and any
fibre $F^{-1}(x)$ by $R_{M},R_{N}$ and $\hat{R},$ respectively.

Let $F$ be a conformal anti-invariant submersion from a K\"{a}hler manifold
$(M,g,J)$ to a Riemannian manifold $(N,g^{\prime})$. We denote by $K$ the
sectional curvature, defined for any pair of non zero orthogonal vectors $X$
and $Y$ on $M$ by the formula:%
\begin{equation}
K(X,Y)=\frac{R(X,Y,Y,X)}{\Vert X\Vert^{2}\Vert Y\Vert^{2}}. \label{e.q:5.1}%
\end{equation}
We denote sectional curvatures of $M$, $N$ and any fibre $F^{-1}(x)$ by
$K_{M},K_{N}$ and $\hat{K}$ respectively.

\begin{theorem}
\label{teo5}Let $F$ be a conformal anti-invariant submersion from a K\"{a}hler
manifold $(M,g,J)$ to a Riemannian manifold $(N,g^{\prime})$ and let
$K_{M},\hat{K}$ and $K_{N}$ be the sectional curvatures of the total space
$M,$ fibers and the base space $N,$ respectively. If $X,Y,Z,H$ are horizontal
and $U,V,W,F$ vertical vectors, then%
\begin{align}
K_{M}(U,V)  &  =\frac{1}{\lambda^{2}}K_{N}(JU,JV)-\frac{3}%
{4}\parallel \mathcal{V}\left[  JU,JV\right]  \parallel ^{2}-\frac
{\lambda^{2}}{2}\{g(\nabla_{JU}\operatorname{grad}(\frac{1}{\lambda^{2}%
}),JU)\nonumber\\
&  +g(\nabla_{JV}\operatorname{grad}(\frac{1}{\lambda^{2}}),JV)\}+\frac
{\lambda^{4}}{4}\{\parallel \operatorname{grad}(\frac{1}{\lambda^{2}})\parallel^{2}\nonumber\\
&  +\parallel
JU(\frac{1}{\lambda^{2}})JV-JV(\frac{1}{\lambda^{2}})JU\parallel^{2}\},\label{e.q:5.2}
\end{align}%
\begin{align}
K_{M}(X,Y)  &  =\hat{K}(BX,BY)+\frac{1}{\lambda^{2}}%
K_{N}(CX,CY)-\frac{3}{4}\parallel \mathcal{V}\left[  CX,CY\right]
\parallel ^{2}\nonumber\\
&  +\frac{\lambda^{2}}{2}\{g(CX,CY)g(\nabla_{CY}\operatorname{grad}(\frac
{1}{\lambda^{2}}),CX)\nonumber\\
&  -g(CY,CY)g(\nabla_{CX}\operatorname{grad}(\frac{1}{\lambda^{2}%
}),CX)+g(CY,CX)g(\nabla_{CX}\operatorname{grad}(\frac{1}{\lambda^{2}%
}),CY)\nonumber\\
&  -g(CX,CX)g(\nabla_{CY}\operatorname{grad}(\frac{1}{\lambda^{2}%
}),CY)\}\nonumber\\
&  +\frac{\lambda^{4}}{4}%
\{(g(CX,CX)g(CY,CY)-g(CY,CX)g(CX,CY))\parallel \operatorname{grad}(\frac{1}{\lambda
^{2}})\parallel^{2}\nonumber\\
&
+\parallel CX(\frac{1}{\lambda^{2}})CY-CY(\frac{1}{\lambda^{2}})CX\parallel^{2}\}+\parallel
T_{BX}BX\parallel ^{2}-g(T_{BY}BY,T_{BX}BX)\nonumber\\
&  +g((\nabla_{BX}A)_{CY}CY,BX)+\parallel A_{CY}BX\parallel ^{2}%
-g((\nabla_{CY}T)_{BX}CY,BX)\nonumber\\
&  -\parallel T_{BX}CY\parallel
^{2}+g((\nabla_{BY}A)_{CX}CX,BY)+\parallel
A_{CX}BY\parallel ^{2}\nonumber\\
&  -g((\nabla_{CX}T)_{BY}CX,BY)-\parallel T_{BY}CX\parallel
^{2}\label{e.q:5.3}
\end{align}
and
\begin{align}
K_{M}(X,U)  &  =\frac{1}{\lambda^{2}}K_{N}(CX,JU)-\frac{3}%
{4}\parallel \mathcal{V}\left[  CX,JU\right]  \parallel ^{2}
\nonumber\\
&  -\frac{\lambda^{2}}{2}\{g(CX,CX)g(\nabla_{JU}\operatorname{grad}(\frac
{1}{\lambda^{2}}),JU)\nonumber\\
&  +g(\nabla_{CX}\operatorname{grad}(\frac{1}{\lambda^{2}}),CX)\}+\frac
{\lambda^{4}}{4}\{g(CX,CX)\parallel \operatorname{grad}(\frac{1}{\lambda^{2}}%
)\parallel^{2}\nonumber\\
&  +\parallel CX(\frac{1}{\lambda^{2}})JU-JU(\frac{1}{\lambda^{2}})CX \parallel^{2}%
\}+g((\nabla_{BX}A)_{JU}JU,BX)+\parallel A_{JU}BX\parallel ^{2}\nonumber\\
&  -g((\nabla_{JU}T)_{BX}JU,BX)-\parallel T_{BX}JU\parallel
^{2}.\label{e.q:5.4}
\end{align}
\end{theorem}

\begin{proof}
 Since $M$ is a K\"{a}hler manifold, we have
$K_{M}(U,V)=K_{M}(JU,JV).
$
Considering (\ref{e.q:2.11}) and (\ref{e.q:5.1}), we obtain
\begin{align*}
K_{M}(U,V)  &  =K_{M}(JU,JV)=g(R_{M}(JU,JV)JV,JU)=\frac{1}{\lambda^{2}%
}g^{\prime}(R_{N}(JU,JV)JV,JU)\\
&  +\frac{1}{4}\{g(\mathcal{V}\left[  JU,JV\right]  ,\mathcal{V}\left[
JV,JU\right]  )-g(\mathcal{V}\left[  JV,JV\right]  ,\mathcal{V}\left[
JU,JU\right]  )\\
&  +2g(\mathcal{V}\left[  JU,JV\right]  ,\mathcal{V}\left[  JV,JU\right]
)\}\\
&  +\frac{\lambda^{2}}{2}\{g(JU,JV)g(\nabla_{JV}\operatorname{grad}(\frac
{1}{\lambda^{2}}),JU)-g(JV,JV)g(\nabla_{JU}\operatorname{grad}(\frac
{1}{\lambda^{2}}),JU)\\
&  +g(JV,JU)g(\nabla_{JU}\operatorname{grad}(\frac{1}{\lambda^{2}%
}),JV)-g(JU,JU)g(\nabla_{JV}\operatorname{grad}(\frac{1}{\lambda^{2}}),JV)\}\\
&  +\frac{\lambda^{4}}{4}%
\{(g(JU,JU)g(JV,JV)-g(JV,JU)g(JU,JV))\parallel
\operatorname{grad}(\frac{1}{\lambda
^{2}})\parallel ^{2}\\
&  +g(JU(\frac{1}{\lambda^{2}})JV-JV(\frac{1}{\lambda^{2}})JU,JU(\frac
{1}{\lambda^{2}})JV-JV(\frac{1}{\lambda^{2}})JU)\}
\end{align*}
for unit vector fields $U$ and $V$. By straightforward computations,
we get (\ref{e.q:5.1}).

 For unit vector fields $X$ and $Y,$ since $M$ is a K\"{a}hler
manifold
and using (\ref{e.q:3.2}), we have%
\begin{align}
K_{M}(X,Y)  &  =K_{M}(JX,JY)=K_{M}(BX,BY)+K_{M}(CX,CY)\label{e.q:5.5}\\
&  +K_{M}(BX,CY)+K_{M}(CX,BY).\nonumber
\end{align}
Using (\ref{e.q:2.11}), we derive%
\begin{eqnarray}
&&K_{M}(BX,BY)    =g(R_{M}(BX,BY)BY,BX)=g(\hat{R}(BX,BY)BY,BX)\nonumber\\
&&  +g(T_{BX}BY,T_{BY}BX)-g(T_{BY}BY,T_{BX}BX)\nonumber\\
&& =\hat{K}(BX,BY)+\parallel T_{BX}BY\parallel
^{2}-g(T_{BY}BY,T_{BX}BX).\label{e.q:5.6}
\end{eqnarray}
In a similar way, using (\ref{e.q:2.14}), we arrive at%
\begin{eqnarray*}
&&K_{M}(CX,CY)   =g(R_{M}(CX,CY)CY,CX)=\frac{1}{\lambda^{2}}g^{\prime}%
(R_{N}(CX,CY)CY,CX)\\
&&  +\frac{1}{4}\{g(\mathcal{V}\left[  CX,CY\right]  ,\mathcal{V}\left[
CY,CX\right]  )-g(\mathcal{V}\left[  CY,CY\right]  ,\mathcal{V}\left[
CX,CX\right]  )\\
&&  +2g(\mathcal{V}\left[  CX,CY\right]  ,\mathcal{V}\left[  CY,CX\right]
)\}\\
&&  +\frac{\lambda^{2}}{2}\{g(CX,CY)g(\nabla_{CY}\operatorname{grad}(\frac
{1}{\lambda^{2}}),CX)-g(CY,CY)g(\nabla_{CX}\operatorname{grad}(\frac
{1}{\lambda^{2}}),CX)\\
&& +g(CY,CX)g(\nabla_{CX}\operatorname{grad}(\frac{1}{\lambda^{2}%
}),CY)-g(CX,CX)g(\nabla_{CY}\operatorname{grad}(\frac{1}{\lambda^{2}}),CY)\}\\
&& +\frac{\lambda^{4}}{4}%
\{(g(CX,CX)g(CY,CY)-g(CY,CX)g(CX,CY))\parallel
\operatorname{grad}(\frac{1}{\lambda
^{2}})\parallel ^{2}\\
&& +g(CX(\frac{1}{\lambda^{2}})CY-CY(\frac{1}{\lambda^{2}})CX,CX(\frac
{1}{\lambda^{2}})CY-CY(\frac{1}{\lambda^{2}})CX)\}.
\end{eqnarray*}
Also by direct calculations, we obtain
\begin{eqnarray}
&&K_{M}(CX,CY)
=\frac{1}{\lambda^{2}}K_{N}(CX,CY)-\frac{3}{4}\parallel
\mathcal{V}\left[  CX,CY\right]  \parallel ^{2}\nonumber\\
&&  +\frac{\lambda^{2}}{2}\{g(CX,CY)g(\nabla_{CY}\operatorname{grad}(\frac
{1}{\lambda^{2}}),CX)-g(CY,CY)g(\nabla_{CX}\operatorname{grad}(\frac
{1}{\lambda^{2}}),CX)\nonumber\\
&&  +g(CY,CX)g(\nabla_{CX}\operatorname{grad}(\frac{1}{\lambda^{2}%
}),CY)-g(CX,CX)g(\nabla_{CY}\operatorname{grad}(\frac{1}{\lambda^{2}%
}),CY)\}\nonumber\\
&&  +\frac{\lambda^{4}}{4}%
\{(g(CX,CX)g(CY,CY)-g(CY,CX)g(CX,CY))\parallel \operatorname{grad}(\frac{1}{\lambda^{2}%
})\parallel ^{2}\nonumber\\
&&  +\parallel CX(\frac{1}{\lambda^{2}})CY-CY(\frac{1}{\lambda^{2}})CX,CX(\frac
{1}{\lambda^{2}})CY\parallel^2\}.\label{e.q:5.7}
\end{eqnarray}
In a similar way, using (\ref{e.q:2.1}) we have
\begin{align}
K_{M}(BX,CY)  &  =g(R_{M}(BX,CY)CY,BX)=g((\nabla_{BX}A)_{CY}CY,BX)+\parallel A_{CY}%
BX\parallel ^{2}\nonumber\\
& -g((\nabla_{CY}T)_{BX}CY,BX)-\parallel T_{BX}CY\parallel
^{2}.\label{e.q:5.8}
\end{align}
Lastly, since $M$ is a K\"{a}hler manifold and using (\ref{e.q:2.13}) we
obtain%
\begin{align}
K_{M}(CX,BY)  &  =K_{M}(BY,CX)=g(R_{M}(BY,CX)CX,BY)=g((\nabla_{BY}%
A)_{CX}CX,BY)\nonumber\\
& +\parallel A_{CX}BY\parallel
^{2}-g((\nabla_{CX}T)_{BY}CX,BY)-\parallel T_{BY}CX\parallel
^{2}.\label{e.q:5.9}
\end{align}
Writing  (\ref{e.q:5.6}), (\ref{e.q:5.7}), (\ref{e.q:5.8}) and
(\ref{e.q:5.9}) in (\ref{e.q:5.5}) we get (\ref{e.q:5.3}).

 For unit vector fields $X$ and $U,$ since $M$ is a
K\"{a}hler manifold
and from (\ref{e.q:3.2}), we have%
\begin{equation}
K_{M}(X,U)=K_{M}(JX,JU)=K_{M}(BX,JU)+K_{M}(CX,JU). \label{e.q:5.10}%
\end{equation}
Using (\ref{e.q:2.13}), we get
\begin{eqnarray}
&&K_{M}(BX,JU)    =g(R_{M}(BX,JU)JU,BX)=g((\nabla_{BX}A)_{JU}JU,BX)+\parallel A_{JU}%
BX\parallel ^{2}\label{e.q:5.11}\\
&& -g((\nabla_{JU}T)_{BX}JU,BX)-\parallel T_{BX}JU\parallel
^{2}.\nonumber
\end{eqnarray}
In a similar way, using (\ref{e.q:2.14}) we obtain%
\begin{eqnarray}
&&K_{M}(CX,JU)  =g(R_{M}(CX,JU)JU,CX)=\frac{1}{\lambda^{2}}g^{\prime}%
(R_{N}(CX,JU)JU,CX)\nonumber\\
&&  +\frac{1}{4}\{g(\mathcal{V}\left[  CX,JU\right]  ,\mathcal{V}\left[
JU,CX\right]  )-g(\mathcal{V}\left[  JU,JU\right]  ,\mathcal{V}\left[
CX,CX\right]  )\nonumber\\
&&  +2g(\mathcal{V}\left[  CX,JU\right]  ,\mathcal{V}\left[  JU,CX\right]
)\}\nonumber\\
&&  +\frac{\lambda^{2}}{2}\{g(CX,JU)g(\nabla_{JU}\operatorname{grad}(\frac
{1}{\lambda^{2}}),CX)-g(JU,JU)g(\nabla_{CX}\operatorname{grad}(\frac
{1}{\lambda^{2}}),CX)\nonumber\\
&&  +g(JU,CX)g(\nabla_{CX}\operatorname{grad}(\frac{1}{\lambda^{2}%
}),JU)-g(CX,CX)g(\nabla_{JU}\operatorname{grad}(\frac{1}{\lambda^{2}%
}),JU)\}\nonumber\\
&&  +\frac{\lambda^{4}}{4}%
\{(g(CX,CX)g(JU,JU)-g(JU,CX)g(CX,JU))\parallel
\operatorname{grad}(\frac{1}{\lambda
^{2}})\parallel ^{2}\nonumber\\
&&  +\parallel CX(\frac{1}{\lambda^{2}})JU-JU(\frac{1}{\lambda^{2}})CX \parallel^2 \}.\label{e.q:5.12}
\end{eqnarray}
If we write (\ref{e.q:5.11}) and (\ref{e.q:5.12}) in
(\ref{e.q:5.10}) and arranging the equation, we get (\ref{e.q:5.4}).
\end{proof}

From Theorem \ref{teo5}, we have the following results.

\begin{corollary}
\label{cor3}Let $F$ be a conformal anti-invariant submersion from a K\"{a}hler
manifold $(M,g,J)$ to a Riemannian manifold $(N,g^{\prime})$. Then we have,%
\begin{align*}
\hat{K}(U,V)  &  \leq\frac{1}{\lambda^{2}}K_{N}(JU,JV)-\frac{\lambda^{2}}%
{2}\{g(\nabla_{JU}\operatorname{grad}(\frac{1}{\lambda^{2}}),JU)+g(\nabla
_{JV}\operatorname{grad}(\frac{1}{\lambda^{2}}),JV)\}\\
&  +\frac{\lambda^{4}}{4}\{\parallel \operatorname{grad}(\frac{1}{\lambda^{2}}%
)\parallel ^{2}+\parallel JU(\frac{1}{\lambda^{2}})JV-JV(\frac{1}{\lambda^{2}})JU\parallel ^{2}%
\}+g(T_{V}V,T_{U}U)
\end{align*}
for $U,V\in\Gamma(\ker F_{\ast})$. The equality case is satisfied if and only
if the fibers are totally geodesic and $J\ker F_{\ast}$ is integrable.
\end{corollary}

\begin{proof}
From (\ref{e.q:5.2}), we have%
\begin{align*}
K_{M}(U,V)  &
=\frac{1}{\lambda^{2}}K_{N}(JU,JV)-\frac{3}{4}\parallel
\mathcal{V}\left[  JU,JV\right]  \parallel ^{2}-\frac{\lambda^{2}}%
{2}\{g(\nabla_{JU}\operatorname{grad}(\frac{1}{\lambda^{2}}),JU)\\
&  +g(\nabla_{JV}\operatorname{grad}(\frac{1}{\lambda^{2}}),JV)\}+\frac
{\lambda^{4}}{4}\{\parallel \operatorname{grad}(\frac{1}{\lambda^{2}})\parallel ^{2}%
+\parallel
JU(\frac{1}{\lambda^{2}})JV-JV(\frac{1}{\lambda^{2}})JU\parallel
^{2}\}.
\end{align*}
Using (\cite{O}, Corollary 1, page: 465), we get%
\begin{align}
\hat{K}(U,V)+\parallel T_{U}V\parallel ^{2}-g(T_{V}V,T_{U}U)  &  =\frac{1}{\lambda^{2}}%
K_{N}(JU,JV)-\frac{3}{4}\parallel \mathcal{V}\left[  JU,JV\right]
\parallel ^{2}\nonumber\\
&  -\frac{\lambda^{2}}{2}\{g(\nabla_{JU}\operatorname{grad}(\frac{1}%
{\lambda^{2}}),JU)+g(\nabla_{JV}\operatorname{grad}(\frac{1}{\lambda^{2}%
}),JV)\}\nonumber\\
&  +\frac{\lambda^{4}}{4}\{\parallel \operatorname{grad}(\frac{1}{\lambda^{2}}%
)\parallel ^{2}+\parallel JU(\frac{1}{\lambda^{2}})JV-JV(\frac{1}{\lambda^{2}})JU\parallel ^{2}%
\}\label{e:q:5.13}
\end{align}
which gives the assertion.
\end{proof}
We also have the following result.
\begin{corollary}
Let $F$ be a conformal anti-invariant submersion from a K\"{a}hler manifold
$(M,g,J)$ to a Riemannian manifold $(N,g^{\prime})$. Then we have,%
\begin{align*}
\hat{K}(U,V)  &
\geq\frac{1}{\lambda^{2}}K_{N}(JU,JV)-\frac{3}{4}\parallel
\mathcal{V}\left[  JU,JV\right]  \parallel ^{2}-\frac{\lambda^{2}}%
{2}\{g(\nabla_{JU}\operatorname{grad}(\frac{1}{\lambda^{2}}),JU)\\
&  +g(\nabla_{JV}\operatorname{grad}(\frac{1}{\lambda^{2}}),JV)\}-\parallel T_{U}%
V\parallel ^{2}+g(T_{V}V,T_{U}U)
\end{align*}
for $U,V\in\Gamma(\ker F_{\ast})$. The equality case is satisfied if and only
if $F$ is a homotetic submersion.
\end{corollary}
\begin{corollary}
\label{cor4}Let $F$ be a conformal anti-invariant submersion from a K\"{a}hler
manifold $(M,g,J)$ to a Riemannian manifold $(N,g^{\prime})$. Then we have,
\begin{align*}
K_{M}(X,Y)  &  \geq\hat{K}(BX,BY)+\frac{1}{\lambda^{2}}K_{N}%
(CX,CY)-\frac{3}{4}\parallel \mathcal{V}\left[  CX,CY\right]
\parallel
^{2}\\
&  +\frac{\lambda^{2}}{2}\{g(CX,CY)g(\nabla_{CY}\operatorname{grad}(\frac
{1}{\lambda^{2}}),CX)\\
&  -g(CY,CY)g(\nabla_{CX}\operatorname{grad}(\frac{1}{\lambda^{2}%
}),CX)+g(CY,CX)g(\nabla_{CX}\operatorname{grad}(\frac{1}{\lambda^{2}}),CY)\\
&  -g(CX,CX)g(\nabla_{CY}\operatorname{grad}(\frac{1}{\lambda^{2}}),CY)\}\\
&  +\frac{\lambda^{4}}{4}%
\{(g(CX,CX)g(CY,CY)-g(CY,CX)g(CX,CY))\parallel
\operatorname{grad}(\frac{1}{\lambda
^{2}})\parallel ^{2}\}\\
&  -g(T_{BY}BY,T_{BX}BX)-g((\nabla_{CY}T)_{BX}CY,BX)+g((\nabla_{BX}%
A)_{CY}CY,BX)\\
&  -\parallel T_{BX}CY\parallel ^{2}+g((\nabla_{BY}A)_{CX}CX,BY)-g((\nabla_{CX}T)_{BY}%
CX,BY)-\parallel T_{BY}CX\parallel ^{2}%
\end{align*}
for $X,Y\in\Gamma((\ker F_{\ast})^{\perp})$. The equality case is satisfied if
and only if $T_{BX}BX=0,$ $A_{CY}BX=0$ and $CX(\frac{1}{\lambda^{2}%
})CY-CY(\frac{1}{\lambda^{2}})CX=0$ which shows that either $\mu$ \
is one dimensional or $\lambda$ is a constant on $\mu$.
\end{corollary}

\begin{proof}
By direct calculations and using (\ref{e.q:5.3})  we arrive at,%
\begin{eqnarray*}
&&K_{M}(X,Y)-\parallel T_{BX}BX\parallel ^{2}-\parallel
A_{CY}BX\parallel ^{2} -\parallel CX(\frac{1}{\lambda^{2}%
})CY-CY(\frac{1}{\lambda^{2}})CX\parallel ^{2}\\
&& =\hat{K}(BX,BY)
+\frac{1}{\lambda^{2}}K_{N}(CX,CY)-\frac{3}{4}\parallel
\mathcal{V}\left[ CX,CY\right]  \parallel
^{2}\\
&&\frac{\lambda^{2}}{2}\{  -g(CY,CY)g(\nabla_{CX}\operatorname{grad}(\frac{1}{\lambda^{2}%
}),CX)+g(CY,CX)g(\nabla_{CX}\operatorname{grad}(\frac{1}{\lambda^{2}}),CY)\\
&& -g(CX,CX)g(\nabla_{CY}\operatorname{grad}(\frac{1}{\lambda^{2}}),CY)+g(CX,CY)g(\nabla
_{CY}\operatorname{grad}(\frac{1}{\lambda^{2}}),CX)\}\\
&& +\frac{\lambda^{4}}{4}%
\{(g(CX,CX)g(CY,CY)-g(CY,CX)g(CX,CY))\parallel
\operatorname{grad}(\frac{1}{\lambda
^{2}})\parallel ^{2}\}\\
&& -g(T_{BY}BY,T_{BX}BX)-g((\nabla_{CY}T)_{BX}CY,BX)+g((\nabla_{BX}%
A)_{CY}CY,BX)\\
&& -\parallel T_{BX}CY\parallel ^{2}+g((\nabla_{BY}A)_{CX}CX,BY)-g((\nabla_{CX}T)_{BY}%
CX,BY)-\parallel T_{BY}CX\parallel ^{2}\\
&& +\parallel A_{CX}BY\parallel^2.
\end{eqnarray*}
This gives the inequality. For the equality case $\parallel T_{BX}BX\parallel ^{2}+\parallel A_{CY}BX\parallel ^{2}+\parallel CX(\frac{1}{\lambda^{2}%
})CY-CY(\frac{1}{\lambda^{2}})CX\parallel ^{2}=0$. Hence we obtain
$T_{BX}BX=0,$ $A_{CY}BX=0$ and
$CX(\frac{1}{\lambda^{2}})CY-CY(\frac{1}{\lambda^{2}})CX=0$ which
shows that either $\mu$ \ is one dimensional or $\lambda$ is a
constant on $\mu$.
\end{proof}
In a similar way, we have the following result.
\begin{corollary}
Let $F$ be a conformal anti-invariant submersion from a K\"{a}hler manifold
$(M,g,J)$ to a Riemannian manifold $(N,g^{\prime})$. Then we have,%
\begin{align*}
\text{ }K_{M}(X,Y)  &  \leq\hat{K}(BX,BY)+\frac{1}{\lambda^{2}}K_{N}%
(CX,CY)-\frac{\lambda^{2}}{2}\{g(CX,CY)g(\nabla_{CY}\operatorname{grad}%
(\frac{1}{\lambda^{2}}),CX)\\
&  -g(CY,CY)g(\nabla_{CX}\operatorname{grad}(\frac{1}{\lambda^{2}%
}),CX)+g(CY,CX)g(\nabla_{CX}\operatorname{grad}(\frac{1}{\lambda^{2}}),CY)\\
&  -g(CX,CX)g(\nabla_{CY}\operatorname{grad}(\frac{1}{\lambda^{2}}),CY)\}\\
&  +\frac{\lambda^{4}}{4}%
\{(g(CX,CX)g(CY,CY)-g(CY,CX)g(CX,CY))\parallel
\operatorname{grad}(\frac{1}{\lambda
^{2}})\parallel ^{2}\\
&  +\parallel CX(\frac{1}{\lambda^{2}})CY-CY(\frac{1}{\lambda^{2}})CX\parallel ^{2}%
\}+\parallel T_{BX}BX\parallel ^{2}-g(T_{BY}BY,T_{BX}BX)\\
&  +g((\nabla_{BX}A)_{CY}CY,BX)+\parallel A_{CY}BX\parallel ^{2}-g((\nabla_{CY}T)_{BX}CY,BX)\\
&  +g((\nabla_{BY}A)_{CX}CX,BY)+\parallel A_{CX}BY\parallel ^{2}-g((\nabla_{CX}T)_{BY}%
CX,BY)-\parallel T_{BY}CX\parallel ^{2}%
\end{align*}
for $X,Y\in\Gamma((\ker F_{\ast})^{\perp})$. The equality case is satisfied if
and only if $T_{BX}CY=0$ and $[CX,CY]\in \Gamma(\mathcal{H})$.
\end{corollary}

\begin{corollary}
Let $F$ be a conformal anti-invariant submersion from a K\"{a}hler manifold
$(M,g,J)$ to a Riemannian manifold $(N,g^{\prime})$. Then we have,%
\begin{align*}
K_{M}(X,U)  &
\geq\frac{1}{\lambda^{2}}K_{N}(CX,JU)-\frac{3}{4}\parallel
\mathcal{V}\left[  CX,JU\right]  \parallel ^{2}\\
&  -\frac{\lambda^{2}}{2}\{g(CX,CX)g(\nabla_{JU}\operatorname{grad}(\frac
{1}{\lambda^{2}}),JU)\\
&  +g(\nabla_{CX}\operatorname{grad}(\frac{1}{\lambda^{2}}),CX)\}+g((\nabla
_{BX}A)_{JU}JU,BX)\\
&  -g((\nabla_{JU}T)_{BX}JU,BX)-\parallel T_{BX}JU\parallel ^{2}%
\end{align*}
for $X\in\Gamma((\ker F_{\ast})^{\perp})$ and $U\in\Gamma(\ker
F_{\ast})$. The equality case is satisfied if and only if
$A_{JU}BX=0$, $\operatorname{grad}(\frac{1}{\lambda^{2}})=0$ and $F$
horizontally homothetic submersion.
\end{corollary}

\begin{proof}
By straightforward computations and using (\ref{e.q:5.4})  we obtain,%
\begin{eqnarray*}
&&K_{M}(X,U)-\parallel A_{JU}BX\parallel ^{2}-\frac{\lambda^{4}}{4}\{g(CX,CX)\parallel \operatorname{grad}%
(\frac{1}{\lambda^{2}})\parallel ^{2}\\
&&+\parallel CX(\frac{1}{\lambda^{2}})JU-JU(\frac{1}%
{\lambda^{2}})CX\parallel ^{2}\}=\frac{1}{\lambda^{2}}K_{N}(CX,JU)-\frac{3}%
{4}\parallel \mathcal{V}\left[  CX,JU\right]  \parallel ^{2}\\
&&-\frac{\lambda^{2}}{2}\{g(CX,CX)g(\nabla_{JU}\operatorname{grad}(\frac
{1}{\lambda^{2}}),JU)+g(\nabla_{CX}\operatorname{grad}(\frac{1}{\lambda^{2}%
}),CX)\}\\
&&+g((\nabla_{BX}A)_{JU}JU,BX)-g((\nabla_{JU}T)_{BX}JU,BX)-\parallel
T_{BX}JU\parallel ^{2}.
\end{eqnarray*}
This gives the inequality. For the equality case $\parallel A_{JU}BX\parallel ^{2}+\frac{\lambda^{4}}{4}%
\{g(CX,CX)\parallel
\operatorname{grad}(\frac{1}{\lambda^{2}})\parallel ^{2}+\parallel
CX(\frac {1}{\lambda^{2}})JU-JU(\frac{1}{\lambda^{2}})CX\parallel
^{2}\}=0$. Thus we derive $A_{JU}BX=0$ and
$\operatorname{grad}(\frac{1}{\lambda^{2}})=0,$ $CX(\frac
{1}{\lambda^{2}})JU-JU(\frac{1}{\lambda^{2}})CX=0$ which shows that
$F$ is horizontally homotetic.
\end{proof}
Finally we have the following inequality.
\begin{corollary}
Let $F$ be a conformal anti-invariant submersion from a K\"{a}hler manifold
$(M,g,J)$ to a Riemannian manifold $(N,g^{\prime})$. Then we have,%
\begin{align*}
K_{M}(X,U)  &  \leq\frac{1}{\lambda^{2}}K_{N}(CX,JU)-\frac{\lambda^{2}}%
{2}\{g(CX,CX)g(\nabla_{JU}\operatorname{grad}(\frac{1}{\lambda^{2}%
}),JU)\\
& +g(\nabla_{CX}\operatorname{grad}(\frac{1}{\lambda^{2}}),CX)\}\\
&  +\frac{\lambda^{4}}{4}\{g(CX,CX)\parallel \operatorname{grad}(\frac{1}{\lambda^{2}%
})\parallel ^{2}+\parallel CX(\frac{1}{\lambda^{2}})JU-JU(\frac{1}{\lambda^{2}})CX\parallel ^{2}\}\\
& +g((\nabla_{BX}A)_{JU}JU,BX)+\parallel A_{JU}BX\parallel
^{2}-g((\nabla_{JU}T)_{BX}JU,BX)
\end{align*}
for $X\in\Gamma((\ker F_{\ast})^{\perp})$ and $U\in\Gamma(\ker F_{\ast})$. The
equality case is satisfied if and only if $T_{BX}JU=0$ and $\left[
CX,JU\right]  \in\mathcal{H}$.
\end{corollary}

\end{document}